\documentclass[english,theo]{manuart}
\pdfoutput=1
\pdfmapfile{+frcr.map}

\usepackage{macros_peyre}
\usepackage{notations_peyre}

\makeindex
\makeglossary

\title[Beyond heights]{Beyond heights: slopes\\ and distribution
  of rational points}
\author{Emmanuel Peyre}
\address{Institut Fourier\\
Universit\'e Grenoble Alpes et CNRS\\
CS 40700\\ 38058 Grenoble CEDEX 09\\ France}
\email{Emmanuel.Peyre@univ-grenoble-alpes.fr}
\date{\today}

\begin{document}

\begin{abstract}
The distribution of rational points of bounded height on algebraic varieties is far from uniform. Indeed the points tend to accumulate on thin subsets which are images of non-trivial finite morphisms. The problem is to find a way to characterise the points in these thin subsets. The slopes introduced by Jean-Benoît Bost are a useful tool for this problem. These notes will present several cases in which this approach is fruitful. We shall also describe the notion of locally accumulating subvarieties which arises when one considers rational points of bounded height near a fixed rational point.
\end{abstract}
\ifx\undefined\altabstract\else
\begin{altabstract}
La distribution des points rationnels de hauteur born\'ee sur les vari\'et\'es alg\'ebriques est loin d'\^etre uniforme les points peuvent s'accumuler sur l'image de vari\'et\'es formant un ensemble mince. La difficult\'e est de pouvoir caract\'eriser les points de ces ensembles accumulateurs. Les pentes de la g\'eom\'etrie d'Arakelov forment un outil utile pour attaquer cette probl\'ematique. Ces notes pr\'esenteront diff\'erents exemples où cette approche est efficace. On \'evoquera \'egalement la question des sous-vari\'et\'es localement accumulatrices qui apparaissent lorsqu'on considère les points de hauteur born\'ee au voisinage d'un point rationnel.
\end{altabstract}
\fi

\maketitle
\tableofcontents

\section{Introduction}
For varieties with infinitely many rational points, one may
equip the variety with a height and study asymptotically the finite set
of rational points with a bounded height. The study of many examples
shows that the distribution of rational points of bounded
height on algebraic varieties is far from uniform.
Indeed the points tend to accumulate on thin subsets which
are images of non-trivial finite morphisms. It is natural to
look for new invariants to characterise the points in
these thin subsets. First of all,
it is natural to consider all possible heights, instead of
one relative to a fixed line bundle. But the geometric analogue
described in section~\ref{section:geometricanalogue}
suggests to go beyond heights to find a property similar to being
very free for rational curves.
The slopes introduced by Jean-Benoît Bost
give the tool for such a construction. In section~\ref{section:slopes},
we describe the notion
of freeness which measures how free a rational point is.
This section will present several cases in which
this approach is fruitful. In section~\ref{section:local},
we also describe its use in
connection with the notion of locally accumulating subvarieties
which arises when one considers
rational points of bounded height near a fixed rational point.

The author thanks D.~Loughran for a discussion which led to a crucial
improvement of this paper.
\section{Norms and heights}
\subsection{Adelic metric}
In this chapter, I am going to use heights defined by an adelic metric,
which I use in a more restrictive sense than in the rest of the summer
school. In fact, an adelic metric will be an analog of the notion of
Riemannian metric in the adelic setting. Let me fix some notation for
the remaining of these notes.
\begin{notas}
  The letter $\KK$ denotes a number field. The set of places
  of~$\KK$ is denoted by $\Valde\KK$. Let $w$ be a place of $\KK$.
  We denote by $\KK_w$ the completion of~$\KK$ at~$w$. For an ultrametric
  place, $\mathcal O_w$ is the ring of integers of $\KK_w$
  and $\mathfrak m_w$ its maximal ideal.
  Let~$v\in\Val(\QQ)$
  denote the restriction of~$w$ to~$\QQ$. We consider
  the map $|\cdot|_w:\KK_w\to\RRp$ defined by
  \[|x|_w=|N_{\KK_w/\QQ_v}(x)|_v\]
  for~$x\in\KK_w$, where $N_{\KK_w/\QQ_v}$ denotes the norm map.
  The Haar measure on the locally compact field $\KK_w$
  is normalized as follows:
  \begin{assertions}
  \item
    $\int_{\mathcal O_w}\Haar{x_w}=1$ for a non-archimedean place~$w$;
  \item
    $\Haar{x_w}$ is the usual Lebesgue measure if~$w$ is real;
  \item
    $\Haar{x_w}=2\Haar x\Haar y$ for a complex place.
  \end{assertions}
\end{notas}
\begin{rema}
  The map~$|\cdot|_w$ is an absolute value if~$w$ is ultrametric
  or real, it is the square of the modulus for a complex place.
  This choice of notation is motivated by the fact that $|\lambda|_w$
  is the multiplier of the Haar measure for the change of variables
  $y=\lambda x$:
  \[\Haar{y_w}=|\lambda|_w\Haar{x_w}\]
  and we have the product formula:
  \[\prod_{w\in\Val(\KK)}|x|_w=1\]
  for any $x\in\KK^*$.
\end{rema}
\begin{term}
  We shall say that a variety $V$ is \emph{nice}\index{Nice variety}
  if it is smooth, projective, and geometrically integral.
\end{term}
\begin{nota}
  Let $X$ be a variety over $\KK$.
  For any commutative $\KK$-algebra~$A$, we denote by $X_A$
  the product $X\times_{\Spec(\KK)}\Spec(A)$ and by $X(A)$
  the set of $A$-points which is defined as
  $\Mor_{\Spec(\KK)}(\Spec(A),X)$.
  \par
  For the rest of this chapter,
  we denote by~$V$ a nice variety on the number field~$\KK$.
  The Picard group of~$V$, denoted by
  $\Pic(V)$, is thought as the set of isomorphism classes
  of line bundles on~$V$.
\end{nota}
\begin{defi}
  Let $\pi:E\to V$ be a vector bundle on~$V$.
  For any extension~$\LL$ of~$\KK$ and any $\LL$-point~$P$ of~$V$,
  we denote by $E_P\subset E(\LL)$ the $\LL$-vector space
  corresponding to the fiber $\pi^{-1}(P)$ of~$\pi$ at~$P$.
  In this text,
  a classical adelic norm on~$E$ is a family
  $(\Vert\cdot\Vert_w)_{w\in\Val(\KK)}$ of continuous maps
  \[\Vert\cdot\Vert_w:E(\KK_w)\to\RRp\]
  such that:
  \begin{conditions}
  \item
    If $w$ is non-archimedean, for any $P\in V(\KK_w)$,
    the restriction ${\Vert\cdot\Vert_w}_{|E_P}$
    is an ultrametric norm with values in $\im(|\cdot|_w)$;
  \item
    If $\KK_w$ is isomorphic to $\RR$, then, for any $P$ in $V(\KK_w)$,
    the restriction ${\Vert\cdot\Vert_w}_{|E_P}$ is a euclidean norm;
  \item
    If $\KK_w$ is isomorphic to~$\CC$, then, for any~$P$ in $V(\KK_w)$,
    there exists a positive definite hermitian form $\phi_P$ on $E_P$ such that
    \[\forall y\in E_P,\quad\Vert y\Vert_w=\phi_P(y,y);\]
  \item
    There exists a finite set of places $S\subset\Val(\KK)$ containing
    the set of archimedean places and
    a model $\mathcal E\to\mathcal V$ of $E\to V$ over $\mathcal O_S$
    such that for any place $w$ in $\Val(\KK)\setminus S$
    and any $P\in\mathcal V(\mathcal O_w)$
    \[\mathcal E_P=\{\,y\in E_P\mid\,\Vert y\Vert_w\leq 1\,\},\]
    where $\mathcal E_P$ denotes the $\mathcal O_w$-submodule
    of $E_P$ defined by $\mathcal E$.
  \end{conditions}
  \par
  In the rest of this chapter, we shall say adelic norm for classical
  adelic norm. An \emph{adelically normed vector bundle}%
  \index{Adelically normed} is a vector bundle equipped with an
  adelic norm.
  We call \emph{adelic metric}\index{Adelic metric} an adelic norm on
  the tangent bundle $TV$.
\end{defi}
The point of using this type of norms is that you can
do all the usual constructions:
\begin{listexams}
  \example
  If $E$ and $F$ are vector bundles equipped with classical adelic norms,
  then we can define adelic norms on the dual $E\dual$,
  the direct sum $E\oplus F$ and
  the tensor product $E\otimes F$.
  \example
  If $E$ is a vector bundle equipped with a classical norm, then
  we define a classical norm on the exterior product $\exterieur^mE$
  in the following manner. Let $P\in V(\KK_w)$.
  If $w$ is an ultrametric space, then
  let
  \[\mathcal E_P=\{\,y\in E_P\mid\,\Vert y\Vert_w\leq 1\,\}.\]
  The set $\mathcal E_P$ is a $\mathcal O_w$-submodule
  of $E_P$ of maximal rank. Then we take on $\exterieur^mE_P$
  the norm defined by the module $\exterieur^m\mathcal E_P$.
  In the archimedean case, we choose the norm on $\exterieur^mE_P$
  so that if $(e_1,\dots,e_r)$ is an orthonormal basis of $E_P$ then
  the family
  $(e_{k_1}\wedge e_{k_2}\wedge\dots\wedge e_{k_m})_{1\leq k_1<k2<\dots<k_m\leq r}$
  is an orthonormal basis of $\exterieur^mE_P$.
  \example
  We can define pull-backs for morphisms of nice varieties over $\KK$.
  \example
  If $V=\Spec(\KK)$, then we may consider a vector bundle on $V$
  as a $\KK$-vector space. Let $E$ be a $\KK$ vector space
  of dimension~$r$ equipped
  with an adelic norm $(\Vert\cdot\Vert)_{w\in\Val(\KK)}$.
  Then
  \[\mathcal E=\{\,y\in E\mid\forall w\in\Val(\KK)_f,\Vert y\Vert_w
  \leq 1\,\}\]
  is a projective $\mathcal O_\KK$ module of constant rank $r$.
  \par
  If $r=1$, by the product formula, the product
  \[\prod_{w\in\Val(\KK)}\Vert y\Vert_w\]
  is constant for $y\in E\setminus\{0\}$. So we can define
  \[\dega(E)=-\sum_{w\in\Val(\KK)}\log(\Vert y\Vert_w).\]
  Let $\widehat{\Pic}(\Spec(\KK))$ be the set of isomorphism classes of line
  bundles with an adelic norm on $\Spec(\KK)$. Let $r_1$
  be the number of real places and $r_2$ the number of complex places.
  Let $\Val(\KK)_\infty\subset\Val(\KK)$ be the set of archimedean places.
  Let $H\subset\RR^{\Val(\KK)_\infty}$ be the hyperplane given by the equation
  $\sum_{w\in\Val(\KK)_\infty}X_w=0$. Then the map
  \[x\mapsto (\log(|x|_w))_{w\in\Val(\KK)_\infty}\]
  induces a map from $\mathcal O_K^*$ to $H$.
  Let $\TT$ be the quotient of~$H$ by the image of this map.
  The group $\TT$ is a compact torus of dimension $r_1+r_2-1$
  and we get an exact sequence
  \[0\longrightarrow \TT\longrightarrow\widehat{\Pic}(\Spec(\KK))
  \longrightarrow\Pic(\Spec(\mathcal O_\KK))\times\RR
  \longrightarrow 0\]
  where we map a line bundle~$E$ equipped with an adelic norm
  $(\Vert\cdot\Vert_w)_{w\in\Val(\KK)}$ to the pair
  $([\mathcal E],\dega(E))$ where $[\mathcal E]$ is the class
  of $\mathcal E$ in the ideal class group of $\mathcal O_\KK$.
  \par
  For arbitrary rank~$r$, we may define:
  \[\dega(E)=\dega(\exterieur^r(E)).\]
\end{listexams}

\Subsection{Arakelov heights}
\begin{defi}
  \label{defi:arakelovheight}
  For any vector bundle~$E$ over~$V$ equipped with an adelic norm,
  the corresponding \emph{logarithmic height}%
  \index{Height>logarithmic=(Logarithmic ---)} is defined as
  the map $h_E:V(\KK)\to\RR$ given by $P\mapsto\dega(E_P)$.
  where $E_P$ is the pull-back of $E$ by the map $P:\Spec(\KK)\to V$.
  The corresponding \emph{exponential height}%
  \index{Height>exponential=(Exponential ---)}
  is defined by $H_E=\exp\circ h_E$.
\end{defi}
\begin{rema}
  If $r=\rk(E)$, we have that $h_E=h_{\exterieur^rE}=h_{\det(E)}$.
  Therefore we do not get more than the heights defined by line bundles.
\end{rema}
\begin{exam}\label{exam:projective}
  For any $w\in\Val(\KK)$, we may consider the map
  $\Vert\cdot\Vert_w:\KK_w^{N+1}\to\RR$ defined by
  \[\Vert(y_0,\dots,y_N)\Vert_w=\max_{0\leq i\leq N}|y_i|_w.\]
  This does not define a classical norm on $\KK_w^{N+1}$ in the sense above,
  however it defines a norm on the tautological line bundle as follows.
  Let $w\in \Val(\KK)$. The fibre
  of the tautological $\mathcal O_{\PP^N_\KK}(-1)$
  over a point $P\in\PP^N(\KK_w)$ may be identified
  with the line corresponding to the point. By restricting
  ${\Vert\cdot\Vert_w}$ to these lines, we obtain an adelic norm
  ${(\Vert\cdot\Vert_w)_{w\in\Val(\KK)}}$ on $\mathcal O_{\PP^N_\KK}(-1)$
  and by duality on $\mathcal O_{\PP^N_\KK}(1)$.
  If $(y_0,\dots,y_N)\in\KK^{N+1}\setminus\{0\}$,
  let $P$, also denoted by $[y_0:\dots:y_N]$,
  be the corresponding point in $\PP^N(\KK)$. Then
  $y=(y_0,\dots,y_n)\in\mathcal O(-1)_P$ and we get
  the formula
  \[H_{\mathcal O(-1)}(P)=\prod_{w\in\Val(K)}\Vert y\Vert_w^{-1}.\]
  Thus $H_{\mathcal O(1)}(P)=\prod_{w\in\Val(\KK)}\Vert y\Vert_w$. In
  the case where $\KK=\QQ$ and $y_0,\dots,y_N$ are coprime integers,
  we have $\Vert(y_0,\dots,y_N)\Vert_v=1$ for any finite place~$v$
  and the height may be written as
  \[H_{\mathcal O(1)}(P)=\max_{0\leq i\leq N}|y_i|\]
  which is one of the na\"ive heights for the projective space.
\end{exam}
\begin{nota}
  For any function $H:V(\KK)\to\RR$, any subset $W\subset V(\KK)$ and any
  positive real number $B$, we consider the set
  \[W_{H\leq B}=\{\,P\in V(\KK)|H(P)\leq B\,\}.\]
\end{nota}
Our aim is to study such sets for heights~$H$ as~$B$ goes to infinity.
Let us motivate this study with a few pictures of such sets.
\begin{exams}
  Figure~\ref{figure.plane} represents rational points of bounded height
  in the projective plane. More precisely this drawing represents
  \[\{\,(x,y)\in\QQ^2\mid H_{\mathcal O(1)}(x:y:1)<40,|x|\leq 1\text{ and }
  |y|\leq 1\,\}.\]
  Figure~\ref{figure.product} represents rational points of bounded height
  in the one-sheeted hyperboloid defined by the equation
  $xy=zt$ in $\PP^3_\QQ$:
  \[\{\,P=(x,y)\in \QQ^2\mid
  H_{\mathcal O(1)}(xy:1:x:y)\leq 50, |x|\leq 1\text{ and }|y|\leq 1\,\}.\]
  This quadric is the image of the Segre embedding
  \[([u_1:v_1],[u_2:v_2])\longmapsto[u_1u_2:v_1v_2:u_1v_2:v_1u_2]\]
  and therefore isomorphic to the product $\PP^1_\QQ\times\PP^1_\QQ$.
  \twofigures{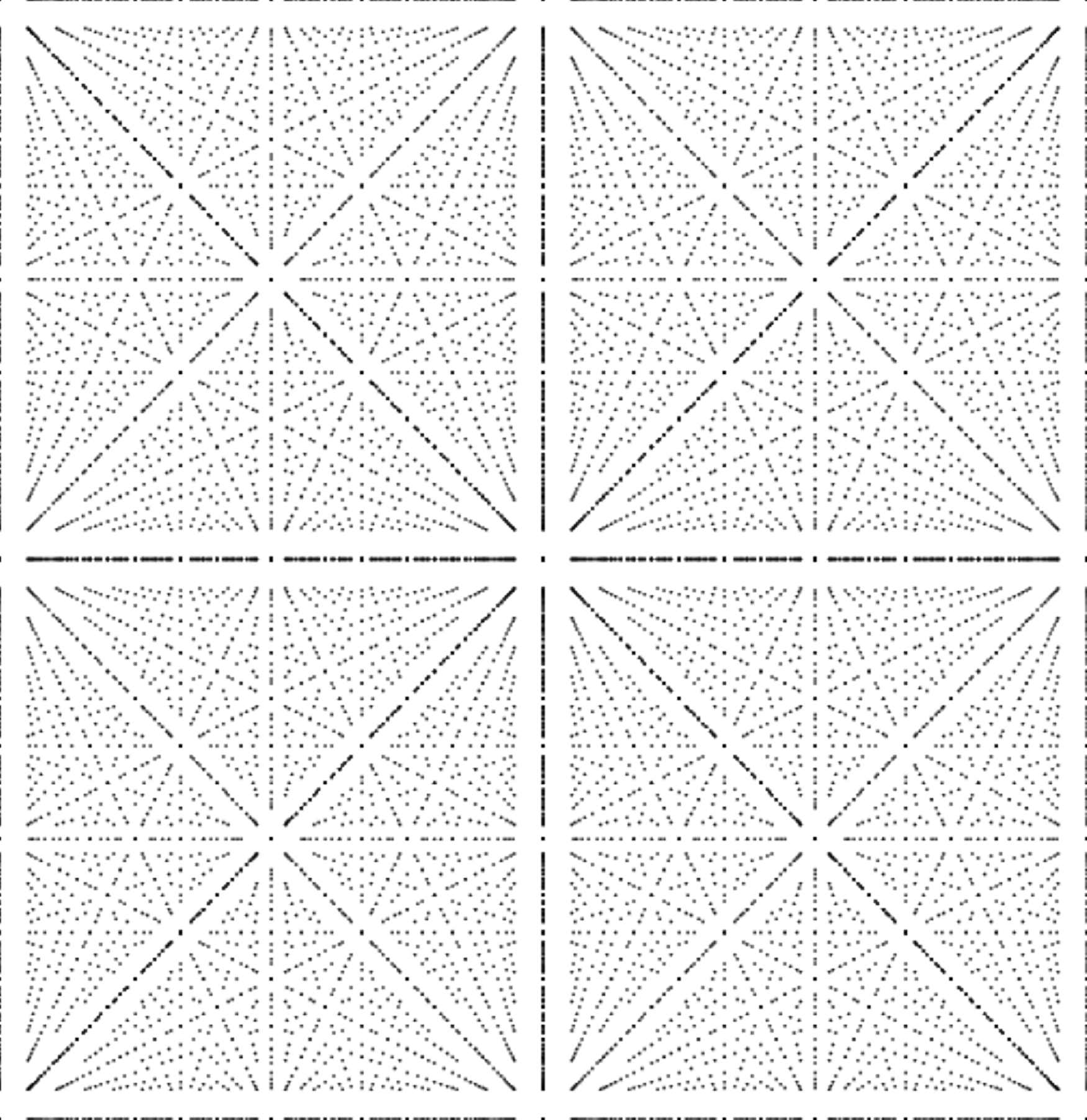}{Projective plane}{figure.plane}%
             {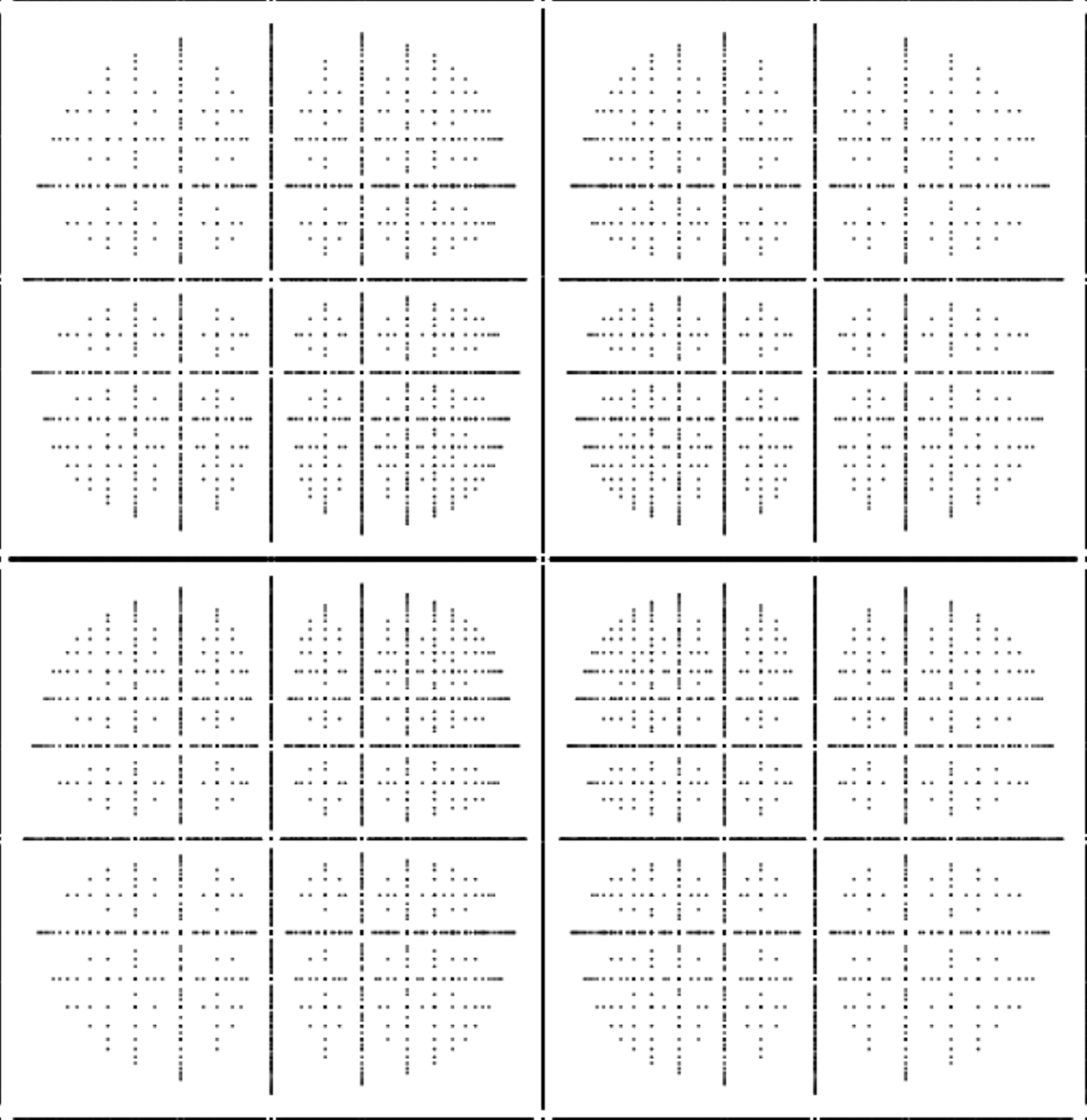}{Hyperboloid}{figure.product}
  The last picture represents rational points of bounded height
  on the sphere:
  \[\{\,P=[x:y:z:t]\in\PP^3(\QQ)|H(P)\leq B\text{ and }x^2+y^2+z^2=t^2\,\}.\]
  \begin{figure}[ht]
    \centering
    \ifcolorcover
    \includegraphics[width=3cm]{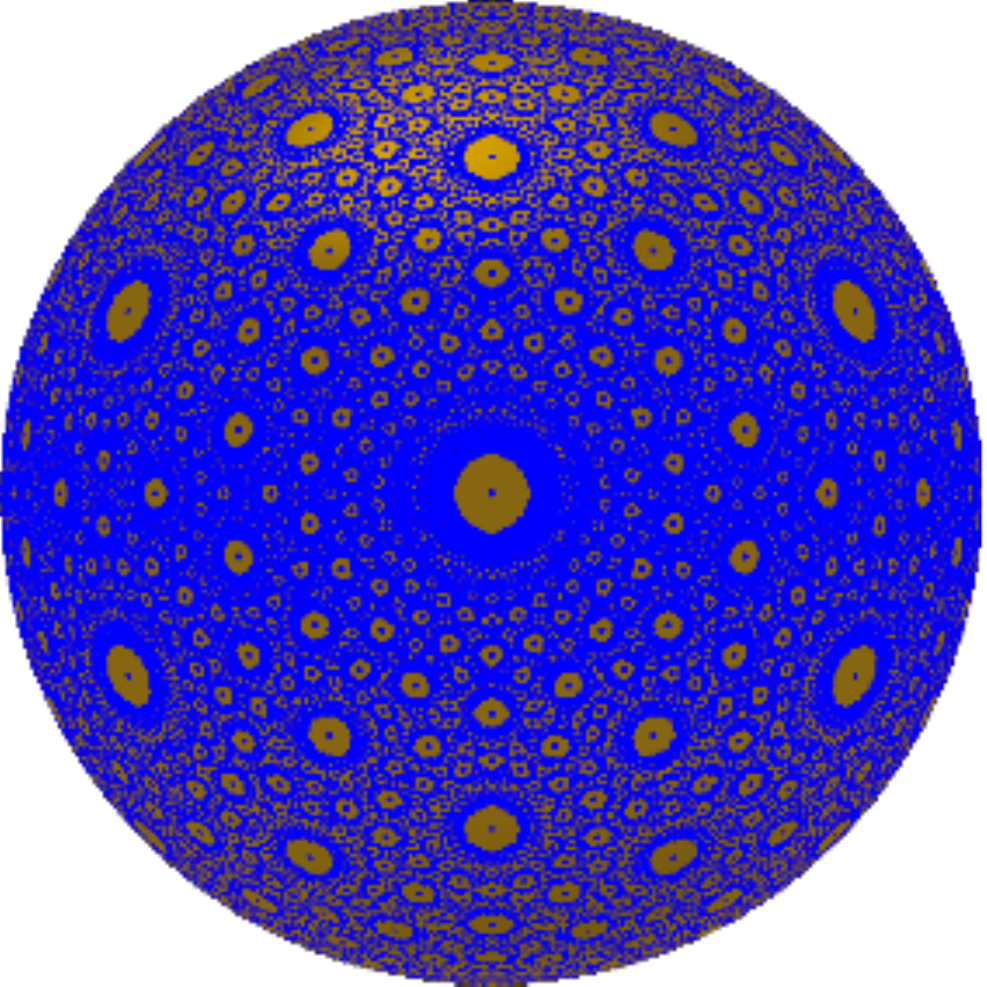}
    \else
    \includegraphics[width=3cm]{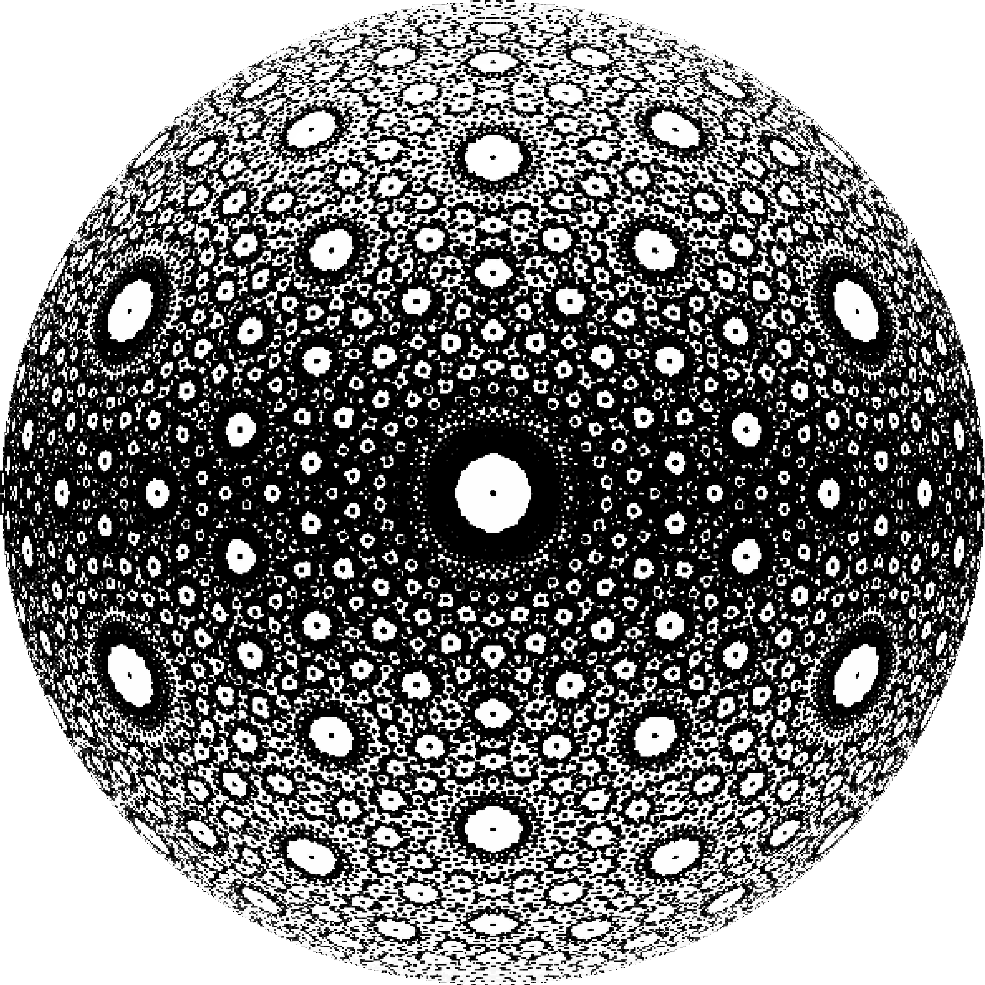}
    \fi
    \caption{The sphere}
    \label{figure.sphere}
  \end{figure}
\end{exams}
\begin{prop}
  If~$L$ is a big line bundle then there exists
  a dense open subset $U\subset V$
  for Zariski topology such that
  for any $B\in\RRp$, the set $U(\KK)_{H\leq B}$ is finite.
\end{prop}
\begin{proof}
  It is enough to prove the result for a multiple of~$L$.
  Thus we may assume that we can write~$L$ as $E+A$
  where~$E$ is effective and~$A$ very ample.
  Taking~$U$ as the complement
  of the base locus of~$E$,
  and choosing a basis $(s_0,\dots,s_N)$ of $\Gamma(V,L)$,
  we get an embedding
  \[U\longrightarrow \PP^N_\KK.\]
  Using the height of example~\ref{exam:projective} on $\PP^N_\KK$,
  we get that
  \[\frac{H(\varphi(x))}{H(x)}=\prod_{w\in\Val(\KK)}\max_{0\leq i\leq N}
  \Vert s_i(x)\Vert_w.\]
  Thus there exists a constant $C\in\RRpp$ such that
  \[\forall x\in V(\KK),\quad H(\varphi(x))\leq CH(x).\]
  Using Northcott theorem, the set of points of bounded height
  in the projective space is finite. A fortiori, the set $U(\KK)_{H\leq B}$
  is finite.
\end{proof}
The height depends on the metric, but in a bounded way:
\begin{prop}
  Let~$H$ and $H'$ be heights defined by adelic norms on a line bundle~$L$
  then the quotient $H/H'$ is bounded: there exist real constants $0<C<C'$
  such that
  \[\forall P\in V(\KK),\quad C\leq\frac{H'(P)}{H(P)}<C'.\]
\end{prop}
\begin{proof}
  The quotient of the norms $\frac{\Vert\cdot\Vert'_w}{\Vert\cdot\Vert_w}$
  induces a continuous map from the compact set $V(\KK_w)$ to $\RRpp$. Thus
  it is bounded from below and above. Moreover the adelic condition
  imposes that the norms coincide for all places outside a finite set.
\end{proof}
\section{Accumulation and equidistribution}
In these notes, I shall first consider the distribution of
rational points
of bounded height on the variety.
\subsection{Sandbox example: the projective space}
First, I have to explain what I mean by distribution.
Let us for example consider the picture in figure~\ref{figure.subset}.
\ifcolorcover
\smallfigure{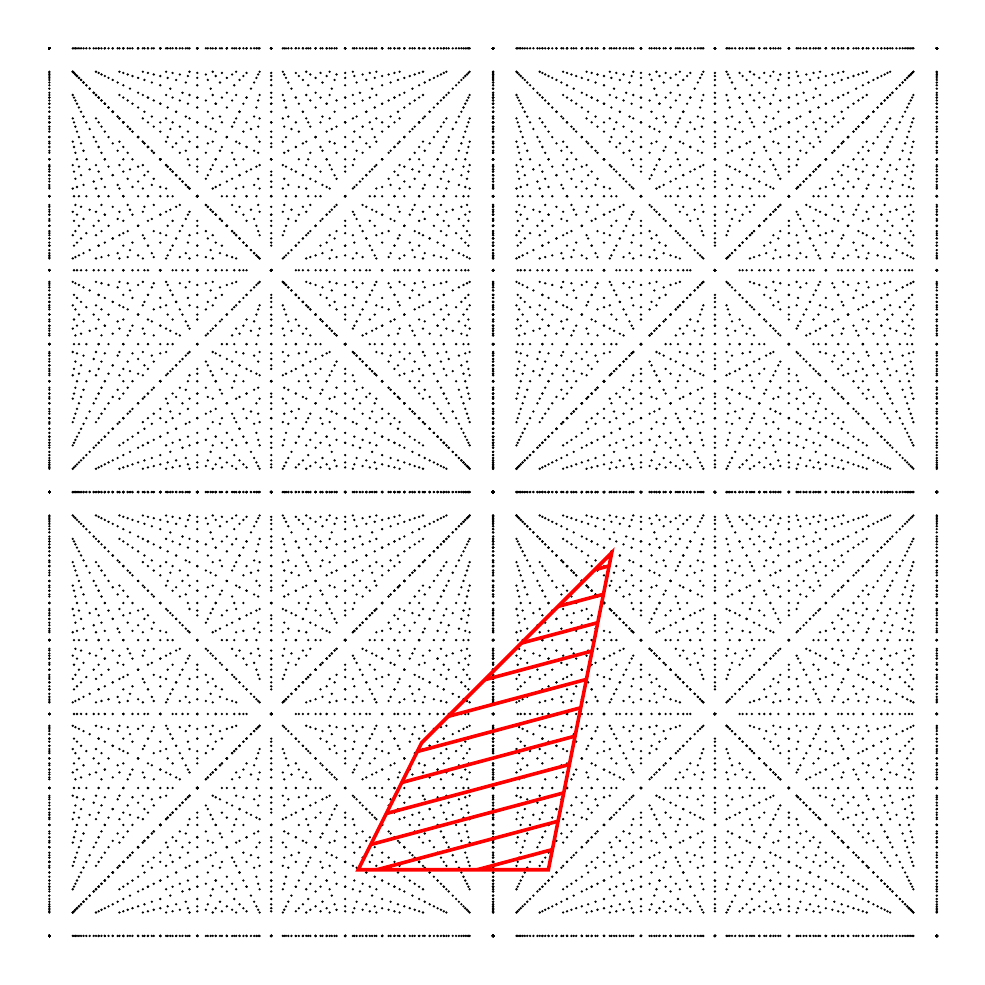}{Open subset}{figure.subset}
\else
\smallfigure{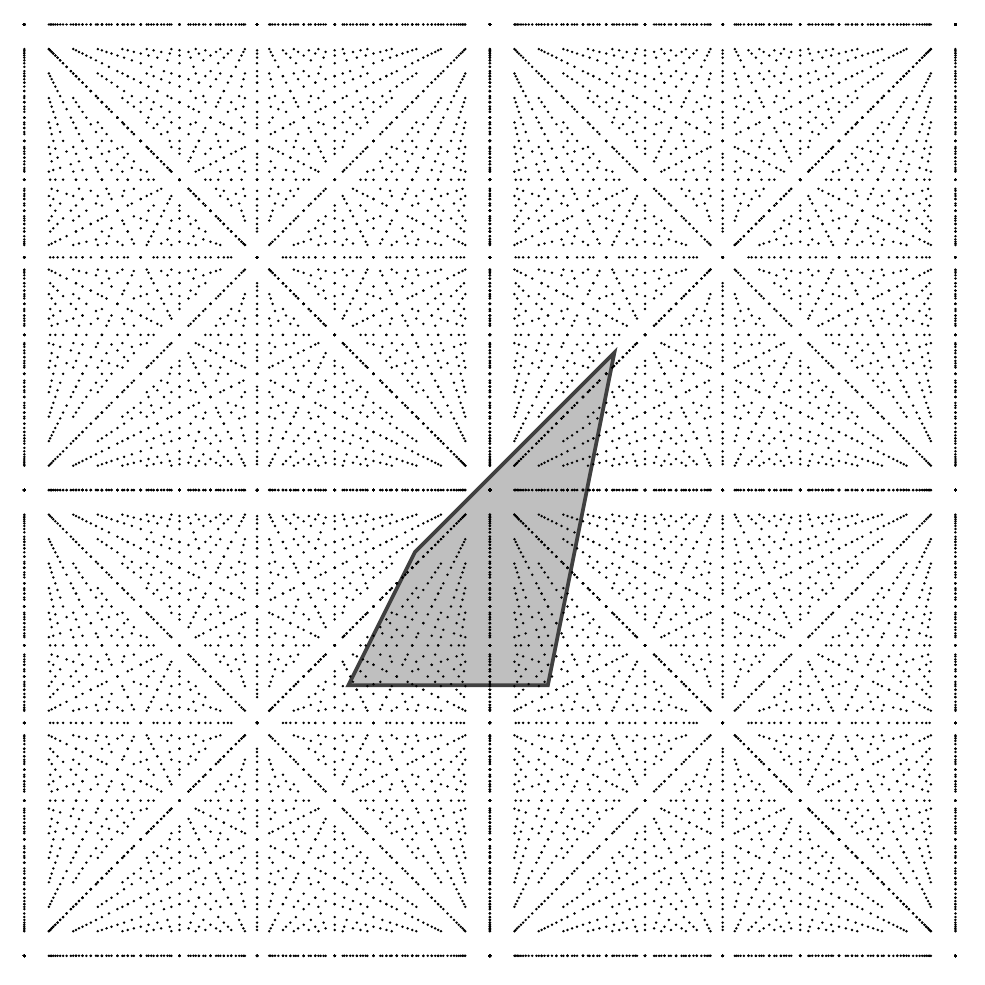}{Open subset}{figure.subset}
\fi
We have selected a ``simple'' open subset~$W$ in $\PP^n(\RR)$, which is
drawn \ifcolorcover with hatching\else in grey\fi.
We may then study asymptotically the proportion
of rational points of bounded height in this open set. More precisely,
one may formulate the following question:
\begin{question}%
  \label{question.distribution.infinite}
  Does the quotient
  \[\frac{\card (W\cap \PP^n(\QQ))_{H\leq B}}{\card \PP^n(\QQ)_{H\leq B}}.\]
  have a limit as $B$ goes to $+\infty$ and how can we interpret its value?
\end{question}

Similarly, let us fix some integer $M>0$ and consider the reduction
modulo $M$ of the points. More precisely, let $A$ be a ring.
The set of $A$-points of the projective space, denoted by $\PP^n(A)$,
is the set of morphisms from $\Spec(A)$ to $\PP^n_\ZZ$. This
defines a covariant functor from the category of rings to the category
of sets.
A $(n+1)$-tuple $(a_0,\dots,a_n)$ in $A^{n+1}$ is said to be primitive
if the generated ideal $(a_0,\dots,a_n)$ is $A$ itself;
this is equivalent to the existence of $(u_0,\dots,u_n)\in A^{n+1}$
such that $\sum_{i=0}^nu_ia_i=1$.
The group of invertible elements acts by multiplication on the
set of primitive elements in $A^{n+1}$.
Then the $\ZZ/M\ZZ$ points of the projective space $\PP_\ZZ^n$ may
be described as the orbits for the action of $\ZZ/M\ZZ^*$ on
the set of primitive elements in $(\ZZ/M\ZZ)^{n+1}$. For any point~$P$
in $\PP^n(\QQ)$, we may choose homogeneous coordinates $[y_0:\dots:y_n]$
so that $y_0,\dots,y_n$ are coprime integers. The reduction modulo $M$
of $P$, is the point of $\PP^n(\ZZ/M\ZZ)$ defined by the primitive
element $(\overline{y_0},\dots,\overline{y_n})$, where $\overline y$
denotes the reduction modulo $M$ of the integer $y$.
This define a map
\[r_M:\PP^n(\QQ)\longrightarrow\PP^n(\ZZ/M\ZZ).\]
This description of the reduction map
generalises easily to any quotient of a principal ring.
Then for any subset $W$ of $\PP^n(\ZZ/M\ZZ)$, we may consider the question
\begin{question}%
  \label{question.distribution.p}
  Does the quotient
  \[\frac{\card (r_M^{-1}(W))_{H\leq B}}{\card \PP^N(\QQ)_{H\leq B}}\]
  converges as~$B$ goes to infinity?
\end{question}

With the adelic point of view, we can see
questions~\ref{question.distribution.infinite}
and~\ref{question.distribution.p} as particular cases
of the following more general question:
\begin{question}%
  \label{question.distribution.adelic}
  Let $\KK$ be a number field.
  Let $\PP^N(\Adeles_\KK)=\prod_{w\in\Val(\KK)}\PP^n(\KK_w)$ be the
  adelic projective space and let $f:\PP^N(\Adeles_\KK)\to\RR$ be a continuous
  function. Does the quotient
  \[S_B(f)=\frac 1{\card\PP^n(\KK)_{H\leq B}}\sum_{P\in\PP^n(\KK)_{H\leq B}}f(P)\]
  have a limit as $B$ goes to infinity?
\end{question}
The answer is postive and we shall state it as a proposition:
\begin{prop}%
  \label{prop:distribution.projective}
  With the notations introduced in question~\ref{question.distribution.adelic},
  \[\CDat S_B(f)@>>B\to+\infty>\int_{\PP^n_\KK(\Adeles_\KK)}f\mmu_{\PP^n}\]
  where $\mmu_{\PP^n}$ is the probability measure defined
  as the product $\prod_{w\in\Val(\KK)}\mmu_w$ where $\mmu_w$ is the
  borelian probability measure on $\PP^n(\KK_w)$ defined by:
  \begin{itemize}
  \item
    If $w$ is a non-archimedean place, let
    $\pi_k:\PP^n(\KK_w)\to\PP^n(\mathcal O_w/\mathfrak m_w^k)$
    be the reduction modulo $\mathfrak m_w^k$ then we equip $\PP^n(\KK_w)$
    with the natural probability measure:
    \[\mmu_w(\pi_k^{-1}(W))=\frac{\card W}{\card
      \PP^n(\mathcal O_w/\mathfrak m_w^k)}\]
    for any subset $W$ of $\PP^n(\mathcal O_w/\mathfrak m_w^k)$;
  \item
    If $w$ is archimedean, let
    $\pi:\KK_w^{n+1}\setminus\{0\}\to\PP^n(\KK_w)$ be the natural
    projection. Than $\mu_w$ is defined by
    \[\mu_w(U)=\frac{\Vol(\pi^{-1}(U)\cap B_{\Vert\cdot\Vert_w}(1))}%
         {\Vol(B_{\Vert\cdot\Vert_w}(1))},\]
    for any borelian subset~$U$ in $\PP^n(\KK_w)$, where
    $B_{\Vert\cdot\Vert_w}(1)$ denotes the ball of radius $1$ for
    $\Vert\cdot\Vert_w$.
  \end{itemize}
\end{prop}
As a consequence, we may give a precise answer to
questions~\ref{question.distribution.infinite}
and~\ref{question.distribution.p}:
\begin{coro}
  If~$W$ is an open subset of $\PP^n(\Adeles_\KK)$ such that
  $\mmu_{\PP^n}(\partial W)=0$ then
  \[\CDat\frac{\card(W\cap\PP^n(\KK))_{H\leq B}}{\card\PP^n(\KK)_{H\leq B}}
  @>>B\to+\infty>\mmu_{\PP^n}(W).\]
\end{coro}
\begin{proof}[Sketch of the proof of
    proposition~\ref{prop:distribution.projective} for $\KK=\QQ$]
  Take an open cube $\mathcal C=\prod_{i=0}^n\leftopen a_i,b_i\rightopen$
  where $a_i$ and $b_i$
  are real numbers with $a_i<b_i$ for $i\in\{0,\dots,n\}$, an integer
  $M\geq 1$ and an element $P_0\in\PP^n(\ZZ/M\ZZ)$.
  We imbed $\RR^n$ in $\PP^n(\RR)$ and consider $\mathcal C$ as an open
  subset of the projective space. We choose
  a primitive element $y_0$ in $(\ZZ/M\ZZ)^{n+1}$ representing $P_0$.
  We then want to estimate
  \begin{align*}
    &\card\{\,P\in\PP^n(\QQ)\mid H(P)\leq B,P\in\mathcal C
    \text{ and }\pi_M(P)=P_0\,\}\\
    &=\frac 12\sum_{\lambda\in\ZZ/M\ZZ^*}
    \card\{\,y\in\ZZ^{n+1}\mid y\text{ primitive},
    \Vert y\Vert_\infty\leq B, y\in\pi^{-1}(\mathcal C)\text{ and }
    y\equiv \lambda y_0\,[M]\,\}\\
    &=\frac12\sum_{\substack{d>0\\\lambda\in\ZZ/M\ZZ^*}}\mu(d)\card\{\,
    y\in(d\ZZ)^{n+1}\setminus\{0\}\mid\Vert y\Vert_\infty\leq B,y\in
    \pi^{-1}(\mathcal C)\text{ and }y\equiv \lambda y_0\,[M]\,\}
  \end{align*}
  where $\mu:\NN\setminus\{0\}\to\{-1,0,1\}$ denotes the Moebius
  function. As $y_0$ is primitive, the set we obtained in the sum is empty
  if $M$ and $d$ are not coprime. Otherwise it is the intersection
  of the translation of a lattice of covolume $(dM)^{n+1}$, the
  cone $\pi^{-1}(\mathcal C)$ and the ball $B_{\Vert\cdot\Vert_\infty}(B)$.
  Thus its cardinal may be approximated by
  \[\frac{\Vol(\pi^{-1}(\mathcal C)\cap B_{\Vert\cdot\Vert_\infty}(1))B^{n+1}}%
         {(dM)^{n+1}}\]
  with an error term which is bounded up to a constant by
  $\Bigl(\frac Bd+1\Bigr)^n$. Up to an error term left to the reader,
  we get that the sum is equivalent to
  \[\frac 12\Vol(\pi^{-1}(\mathcal C)\cap B_{\Vert\cdot\Vert_\infty}(1))
  \times\frac{\varphi(M)}{M^{n+1}\prod_{p\mid M}\Bigl(1-\frac 1{p^{n+1}}\Bigr)}
  \times\frac 1{\zeta_\QQ(n+1)}B^{n+1}.\]
  In this product, the term
  \[\frac{\varphi(M)}{M^{n+1}\prod_{p\mid M}\Bigl(1-\frac 1{p^{n+1}}\Bigr)}\]
  is $\card(\PP^n(\ZZ/M\ZZ))^{-1}$. In particular, this implies that
  ${\frac 12\Vol(B_{\Vert\cdot\Vert}(1))/\zeta_{\QQ}(n+1)}$ is the limit
  of $\card\PP^n(\QQ)_{H\leq B}/B^{n+1}$ as $B$ goes to infinity.
\end{proof}

\subsection{Adelic measure}
By choosing different norms on the anticanonical line bundle,
and thus different heights on a variety, one realizes that the measure
which gives the asymptotic distribution as $B$ goes to infinity may be
directly defined from the adelic norm on $\omega_V^{-1}$, exactly
as a Riemannian metric defines a volume form. This construction in
fact applies to any nice variety equipped with an adelic metric.
\begin{cons}
  \label{cons:measure}
  Let $V$ be a nice variety with a rational point.
  We fix an adelic norm ${(\Vert\cdot\Vert_w)_{w\in\Val(\KK)}}$
  on $\omega_V^{-1}$. The formula for the change of variables
  (see \cite[\S2.2.1]{weil:adeles}) proves that the local measures
  \begin{equation}
    \label{equ:local.measure}
    \left\Vert\frac\partial{\partial x_1}\wedge
    \frac\partial{\partial x_2}\wedge\dots\wedge
    \frac\partial{\partial x_n}\right\Vert_w\Haar{x_{1,w}}\Haar{x_{2,w}}\dots
    \Haar{x_{n,w}},
  \end{equation}
  where $(x_1,\dots,x_n):\Omega\to\KK_w^n$ is a local system
  of coordinates defined on an open subset $\Omega$ of $V(\KK_w)$,
  does not depend on the choice of coordinates; therefore by patching together
  these measures, we get a mesure $\oomega_{V,w}$ on $V(\KK_w)$, which induces
  a probability measure
  \[\mmu_{V,w}=\frac 1{\oomega_{V,w}(V(\KK_w))}\oomega_{V,w}.\]
  Then the product
  \[\mmu_V=\prod_{w\in\Val(\KK)}\mmu_{V,w}\]
  is a probability measure on the adelic space $V(\Adeles_\KK)$.    
\end{cons}
\begin{rema}
  For the projective space, this construction gives the right asymptotic
  distribution for the points of bounded height. So it is natural to try to
  generalise to other varieties. To state precisely our question, we
  introduce the counting measure defined by the set of points of bounded
  height.
\end{rema}
\begin{defi}
  For any non-empty subset $W\subset V(\KK)$ we define, for $B$
  a real number bigger
  than the smallest height of a point of~$W$,
  \[\ddelta_{W_{H\leq B}}=\frac 1{\card W_{H\leq B}}\sum_{P\in W_{H\leq B}}
  \ddelta_P,\]
  where $\ddelta_P$ denotes the Dirac measure at $P$ on the adelic space.
\end{defi}
\begin{namedtheorem}[Na\"ive equidistribution]
  \label{defi:naiveequidistribution}
  We shall say that the \emph{na\"ive equidistribution}%
  \index{Naiveequidistribution=Na\"ive equidistribution}%
  \index{Equidistribution>Naive=(Na\"ive ---)}
  \textup{(NE)}\namedlabel{label:overoptimistic}{NE}
  holds if the measure $\ddelta_{V(\KK)_{H\leq B}}$ converges
  to~$\mmu_V$ as~$B$ goes to infinity for the weak topology.
\end{namedtheorem}
\begin{rema}
  In other words, the na\"ive equidistribution holds if for any continuous
  function $f:V(\Adeles_\KK)\to\RR$, one has the convergence
  \[\CDat\int_{V(\Adeles_\KK)}f\ddelta_{V(\KK)_{H\leq B}}@>>B\to+\infty>
  \int_{V(\Adeles_\KK)}f\mmu_V.\]
\end{rema}
This equidistribution may seem to be overoptimistic and one may wonder
whether there exists any case besides the projective space for which it
is valid.
\begin{theo}
  \label{theo:equid.flag}
  If~$V$ is a generalized flag variety, that is a quotient $G/P$ where~$G$
  is a linear algebraic group and~$P$ a parabolic subgroup of~$G$, then
  \eqref{label:overoptimistic} is true.
\end{theo}
\begin{exam}
  Grassmannian are examples of such flag varieties. 
  Any smooth quadric with a rational point is a generalized flag variety
  for the orthogonal group. Therefore any smooth quadric with a rational point
  satisfies the naive equidistribution.
\end{exam}
\begin{proof}[Tools of the proof of theorem~\ref{theo:equid.flag}]
  To prove this result one may use harmonic analysis on the adelic
  space $G/P(\Adeles_\KK)$ and apply Langland's work on Eisenstein series
  (see \cite[corollaire 6.2.17]{peyre:fano}, \cite{langlands:eisenstein}).
\end{proof}
So we have solved the case of hypersurfaces of degree~2. In higher degrees,
the equidistribution is an easy consequence of the very general result
of Birch \cite{birch:forms} about forms with many variables.
\begin{theo}
  Let $V\subset \PP^n_\QQ$ be a smooth hypersurface of degree $d$ such
  that $V(\Adeles_\QQ)\neq\emptyset$ with $n>(d-1)2^d$, then~$V$ satisfies
  \eqref{label:overoptimistic}.
\end{theo}
\begin{rema}
In fact, it applies to all the cases considered by Birch, that is
for smooth complete intersection of~$m$ hypersurfaces of the same degree~$d$
if $n>m(m+1)(d-1)2^{d-1}$.
\end{rema}

\subsection{Weak approximation}
The first indications of the na\"ivet\'e of \eqref{label:overoptimistic}
appear when one considers obvious consequences of it. Let us
recall the definition of weak approximation:
\begin{defi}
  A nice veriety $V$ satisfies \emph{weak approximation}%
  \index{Weak approximation}
  if the rational points of~$V$ are dense in the adelic
  space $V(\Adeles_\KK)$.
\end{defi}
\begin{listrems}
  \remark
  Let~$V$ be a nice variety with a rational point. If it satisfies the
  na\"ive equidistribution, then it satisfies weak approximation and
  therefore $V(\KK)$ is dense for Zariski topology.

  This follows from the fact that for any real number~$B$, the support
  of the measure $\ddelta_{V(\KK)_{H\leq B}}$ in $V(\Adeles_\KK)$
  is contained in the closure $\overline {V(\KK)}$
  of the set of rational points.
  But the support of the measure $\mmu_V$ is the whole adelic space.
  Thus \eqref{label:overoptimistic} implies that
  $\overline{V(\KK)}=V(\Adeles_\KK)$. For the last statement,
  we use that if~$V$
  has an adelic point then for any place~$w$ of~$\KK$,
  the set $V(\KK_w)$ is Zariski dense in~$V$.
  \remark
  About the density for Zariski topology, we could always reduce to that case
  by considering the desingularisation of the closure of the rational
  points for Zariski topology. Such a reduction is of course not possible when
  the points are dense for Zariski topology but the variety does not satisfy
  weak approximation.
\end{listrems}
\begin{conv}
  From now on, we assume that $V$ is a nice variety in which
  the set of rational points $V(\KK)$ is zariski dense.
\end{conv}
About weak appproximation, we are going to give a quick
overview of the Brauer-Manin
obstruction, which was introduced by Y.~Manin in \cite{manin:brauer}
to explain the previously known counterexamples to weak approximation
(see also \cite{peyre:obstruction} for a survey).
\begin{cons}
  For a nice variety~$V$, we define its \emph{Brauer group}%
  \index{Brauer group} as the cohomology group
  \[\Br(V)=H^2\etale(V,\GG_m)\glossary{Br$(V)$: Brauer group of $V$.}\]
  which defines a contravariant functor from nice varieties
  to the category of abelian groups. In the case of the spectrum
  of a field of characteristic~$0$,
  we get the Brauer group of $\LL$, which is defined in
  terms of Galois cohomology by
  \[\Br(\LL)=H^2(\Gal(\overline\LL/\LL),\GG_m),\]
  where $\overline\LL$ is an algebraic closure of $\LL$.
  Class field theory gives for any place $w$ an injective morphism
  \[\inv_w:\Br(\KK_w)\longrightarrow\QQ/\ZZ\]
  which is an isomorphism if~$w$ is not archimedean, so that the
  sequence
  \begin{equation}
    \label{equ:classfield}
    \CDat 0\to \Br(\KK)\to\bigoplus_{w\in\Val(\KK)}
    \Br(\KK_w)@>\sum_{w}\inv_w>>\QQ/\ZZ\to 0
  \end{equation}
  is an exact sequence.
  Therefore we may define a pairing
  \begin{align*}
    \Br(V)\times V(\Adeles_\KK)&\longrightarrow\QQ/\ZZ\\
    (\alpha,(P_w)_{w\in\Val(\KK)})&
    \longmapsto\sum_{w\in\Val(\KK)}\inv_w(\alpha(P_w))
  \end{align*}
  where $\alpha(P_w)$ denotes the pull-back of $\alpha$
  by the morphism $\Spec(\KK_w)\to V$ defined by~$P_w$.
  Let us denote by $\Br(V)\dual$ the group $\Hom(\Br(V),\QQ/\ZZ)$
  then the above pairing may be seen as a map
  \[\eta:V(\Adeles_\KK)\longrightarrow\Br(V)\dual.\]
  If $P\in V(\KK)$ then the fact that \eqref{equ:classfield}
  is a complex implies that
  \[\sum_{w\in\Val(\KK)}\inv_w(\alpha(P))=0;\]
  in other words, $\eta(P)=0$. By arguments of continuity,
  one gets that
  \[\overline{V(\KK)}\subset V(\Adeles_\KK)^{\Br}
  =\{\,P\in V(\Adeles_\KK)\mid \eta(P)=0\,\}.\]
  The element $\eta(P)$ is called the \emph{Brauer-Manin obstruction}%
  \index{Brauer-Manin obstruction}
  to weak approximation at~$P$.
\end{cons}
\begin{rema}
  Let $\overline\KK$ be an algebraic closure of $\KK$ and
  $\overline V=V_{\overline\KK}$.
  Since we assume $V$ to have a rational point, there is an exact sequence
  \[0\to\Br(\KK)\to\ker(\Br(V)\to\Br(\overline V))\to
  H^1(\Gal(\Kbar/\KK),\Pic(\overline V))\to 0.\]
  Also the exponential map gives an exact sequence
  \begin{multline*}
    H^1(V_\CC,\mathcal O_V)\to\Pic(\overline V)\to H^2(V(\CC),\ZZ)\to\\
    H^2(V_\CC,\mathcal O_V)\to \Br(\overline V)
    \to H^3(V(\CC),\ZZ)_{\tors}
  \end{multline*}
  Thus assuming that $H^i(V,\mathcal O_V)=\{0\}$ for $i=1$ and $i=2$,
  which is automatic
  for Fano varieties by Kodaira's vanishing theorem, we
  get first that the geometric Picard of the variety is finitely generated.
  Thus the action of the Galois group on the Picard group is trivial
  over a finite extension of the ground field.
  Therefore, in this case,
  the groups $H^1(\Gal(\Kbar/\KK),\Pic(\overline V))$ and $\Br(\overline V)$
  are finite.
  Hence the cokernel of the morphism $\Br(\KK)\to\Br(V)$
  is finite, which implies that $V(\Adeles_\KK)^{\Br}$ is open
  and closed in the adelic space.
\end{rema}
If one hopes that the Brauer-Manin obstruction to the weak approximation
is the only one, then it is natural to define the measure induced by the
probability measure $\mmu_V$ on the space on which the obstruction is $0$.
Since we assume that the variety~$V$ has a rational point, the space
$V(\Adeles_\KK)^{\Br}$ is not empty. In that setting, we may give
the following definition:
\begin{defi}
  The measure $\mmu_V^{\Br}$ is defined as follows:
  for any Borelian subset~$W$ of $V(\Adeles_\KK)$
  \[\mmu_V^{\Br}(W)=\frac{\mmu_V(W\cap V(\Adeles_\KK)^{\Br})}%
        {\mmu_V(V(\Adeles_\KK)^{\Br})}.\]
\end{defi}
The following question then takes into account the Brauer-Manin obstruction
to weak approximation:
\begin{namedtheorem}[Global equidistribution]\label{globalequidistribution}
  We shall say that \emph{global equidistribution}%
  \index{Global equidistribution}%
  \index{Equidistribution>Global=(Global ---)}
  holds if the measure $\ddelta_{V(\KK)_{H\leq B}}$ converges
  weakly to $\mmu_V^{\Br}$ as $B$ goes to infinity.
\end{namedtheorem}
Potential counterexamples to global equidistribution have been known for
quite a long time (see for example \cite{serre:mordellweil}),
but Manin was the first to consider accumulating subsets, which we will
study at length in the next section.

\subsection{Accumulating subsets}
In fact, the support of the limit of the measure
$\ddelta_{V(\KK)_{H\leq B}}$ is, in general, much smaller than
the closure $\overline V(\KK)$ of the set of rational points.
Let me give a few examples.

\subsubsection{The plane blown up in one point}
The blowing up of the projective plane at the point $P_0=[0:0:1]$
may be described as the hypersurface~$V$ in the product
$\PP^2_\QQ\times\PP^1_\QQ$ defined by the equation $XV=YU$,
where $X,Y,Z$ denote the coordinates on the first factor
and $U,V$ the coordinates on the second one.
Let $\pi$ be the projection on the first factor. Then
$E=\pi^{-1}(P_0)$ is an exceptional divisor on~$V$ and the
second projection $\pr_2$ defines an isomorphism from~$E$
to $\PP^1_\QQ$. Let~$U$ be the complement of~$E$ in~$V$. The
projection~$\pi$ induces an isomorphism from~$U$
to $\PP^2_\QQ\setminus\{P_0\}$.
As an exponential height, we may use the map
\begin{align*}
  H:V(\QQ)&\longrightarrow\RR_{>0}\\
  (P,Q)&\longmapsto H_{\mathcal O_{\PP^2_\QQ}(1)}(P)^2
  H_{\mathcal O_{\PP^1_\QQ}(1)}(Q).
\end{align*}
This example as been used as a sandbox case for the study
of rational points of bounded height by many people, and the estimation
may be summarized as follows:
\begin{prop}[J.-P Serre, V.~V. Batyrev, Y. I. Manin, et al.]
  On the exceptional line, the number of points of bounded height
  is given by
  \[\card E(\QQ)_{H\leq B}\sim\frac2{\zeta_\QQ(2)}B^2\]
  as $B$ goes to infinity, whereas on its complement it is given by
  \[\card U(\QQ)_{H\leq B}\sim\frac8{3\zeta_\QQ(2)^2}B\log(B)\]
  as $B$ goes to infinity.
\end{prop}
\begin{rema}
  Thus there are much more rational points on the exceptional
  line~$E$ than on the dense open subset~$U$.
  In fact, since the points on the exceptional line are distributed
  as on $\PP^1_\QQ$, we get that the measure $\ddelta_{V(\QQ)_{H\leq B}}$
  converges to $\mmu_E$ for the weak topology.
\end{rema}
On the other hand, if we only consider the rational points on
the open set $U$, we get the right limit:
\begin{prop}
  The measure $\delta_{U(\QQ)_{H\leq B}}$ converges to $\mmu_V$
  for the weak topology as $B$ goes to infinity.  
\end{prop}
\begin{listrems}
  \remark
  Let~$W$ be an infinite subset of $V(\KK)$.
  If the measure $\ddelta_{W_{H\leq B}}$ converges to $\mmu_V$
  for the weak topology, then, for any strict closed subvariety $F$ in $V$,
  we have that
  \[\card(W\cap F(\QQ))_{H\leq B}=o(\card W_{H\leq B})\]
  since we have $\mmu_V(F(\Adeles_\KK))=0$.
  Thus any strict closed subset with a strictly positive contribution
  to the number of points has to be removed to get equidistribution.
  \remark
  It may seem counterintuitive that by removing points, we get a measure
  with a larger support. But this comes from the fact that we divide
  the counting measure on $U$ by a smaller term. From this example, it
  follows that it is natural to consider only the points outside
  a set of ``bad'' points. The problem is that this set of bad points
  might be quite big.
\end{listrems}

\subsubsection{The principle of Manin}
The principle suggested by Manin and his collaborators
in the funding papers \cite{batyrevmanin:hauteur} and
\cite{fmt:fano} is that, on Fano varieties,
there should be an open subset on which the points of bounded
height behave as expected. Let us give a precise expression for this
principle, in a slightly more general setting.
Since this principle deals with the number of points of bounded
height rather than their distribution, we have to introduce
another normalisation of the measures to get a conjectural
value for the constant, which is defined as a volume.
\begin{notas}
  Let $\NS(V)$\glossary{$\text{NS}(V)$: N\'eron-Severi group}
  be the N\'eron-Severi group of~$V$,%
  \index{NeronSeveri group=N\'eron-Severi group}
  that is the quotient
  of the Picard group by the connected component of the
  neutral element. We put $\NS(V)_\RR=\NS(V)\otimes_\ZZ\RR$
  and denote by $\Ceffof V$ the closed cone in $\NS(V)_\RR$
  generated by the classes of effective divisors.
  We write $\Ceffof V\dual$ for the dual of the effective
  cone in the dual space $\NS(V)_\RR\dual$:
  \[\Ceffof V\dual=\{\,y\in\NS(V)_\RR\dual\mid
  \forall x\in\Ceffof V,\langle y,x\rangle\geq 0\,\}.\]
\end{notas}
To construct the constant, we shall restrict ourselves to a setting
in which the local measures can be normalized using the action
of the Galois group of~$\KK$ on the Picard group. Therefore,
we make the following hypothesis:
\begin{hypos}
  \label{hypos:conditions}%
  From now on, $V$ is a nice variety, which satisfies the following conditions:
  \begin{conditions}
  \item A multiple of the class of $\omega_V^{-1}$ is the sum of
    an ample divisor and a divisor with normal crossings;
  \item The set $V(\QQ)$ is Zariski dense;
  \item The groups $H^i(V,\mathcal O_V)$ are $\{0\}$ if $i\in\{1,2\}$;
  \item The geometric Brauer group $\Br(\overline V)$ is trivial and
    the geometric Picard group $\Pic(\overline V)$ has no torsion;
  \item The closed cone
    $\Ceffof {\overline V}$
    is generated by the classes of a finite set of effective
    divisors.
  \end{conditions}
\end{hypos}
\begin{cons}
  \label{cons:tamagawa}
  We choose a finite set~$S$ of places containing the archimedean
  places and the places of bad reduction for~$V$. Let~$\LL$
  be a finite extension of $\KK$ such that the Picard
  group $\Pic(V_\LL)$ is isomorphic to the geometric
  Picard group $\Pic(\overline V)$. We assume that~$S$
  contains all the places which ramify in the extension $\LL/\KK$.
  With this assumption, for any place~$w\in\Val(\KK)\setminus S$,
  let $\FF_w$ be the residual field at $w$. The Frobenius lifts
  to a an element $(w,\LL/\KK)$ in $\Gal(\LL/\KK)$ which is well
  defined up to conjugation (see \cite[\S1.8]{serre:corpslocaux}).
  Then we can consider the local factors of the~$L$ function defined
  by the Picard group:
  \[L_w(s,\Pic(\overline V))=\frac 1{\det(1-\card \FF_w^{-s}(w,\LL/\KK)|
    \Pic(\overline V))},\]
  where $s$ is a complex number with $\Re(s)>0$. If the real part
  of~$s$ satisfies $\Re(s)>1$, then the eulerian product
  \[L_S(s,\Pic(\overline V))=\prod_{w\in\Val(\KK)\setminus S}L_w(s,\overline V)\]
  converges. For $w\in\Val(\KK)$,
  we define $\lambda_w=L_w(1,\Pic(\overline V))^{-1}$ if $w\not\in S$
  and $\lambda_w=1$ otherwise. We put $t=\rk(\Pic(V))$.
  It follows from the Weil's conjecture proven
  by P.~Deligne \cite{deligne:weil} that the product of measures
  \begin{equation}\label{equ:adelicmeasure}
      \oomega_V=\frac{\lim_{s\to 1}(s-1)^tL_S(s,\Pic(\overline V))}%
             {\sqrt{d_\KK}^{\,\dim(V)}}\prod_{w\in\Val(\KK)}\lambda_w\oomega_{V,w}
  \end{equation}
  converges (see \cite[\S2.1]{peyre:fano}). We may then define the
  \emph{Tamagawa-Brauer-Manin volume of $V$} as
  \[\tau^{\Br}(V)=\oomega_V(V(\Adeles_\KK)^{\Br}).\]
  We also introduce the constant
  \[\alpha(V)=\frac1{(t-1)!}\int_{C^1\eff(V)\dual}e^{-\langle\omega_V^{-1},y\rangle}
  \Haar y\]
  which is a rational number under the
  hypothesis~\ref{hypos:conditions}~(v),
  and the integer
  \[\beta(V)=\card(\Br(V)).\]
  Then the \emph{empirical constant} associated to the chosen metric
  on~$V$ is the constant
  \[C(V)=\alpha(V)\beta(V)\tau^{\Br}(V).\]
\end{cons}
\begin{namedtheorem}[Batyrev-Manin principle]
  Let~$V$ be a variety which satisfies the conditions~\ref{hypos:conditions}.
  We say that~$V$ satisfies the
  \emph{refined Batyrev-Manin principle}\index{Batyrev-Manin principle} if
  there exists a dense open subset~$U$ of~$V$ such that
  \begin{equation}
    \label{equ:batyrevmanin}
    \card U(\KK)_{H\leq B}\sim C(V)B\log(B)^{t-1}
  \end{equation}
  as $B$ goes to infinity.
\end{namedtheorem}
For equidistribution, we may introduce the following notion
\begin{namedtheorem}[Relative equidistribution]%
  \label{defi:relativeequidistribution}
  Let $W$ be an infinite subset of $V(\KK)$, we say that
  \emph{the points of~$W$ are equidistributed in~$V$} if the
  counting measure $\ddelta_{W_{H\leq B}}$ converges to~$\mmu_V$.
\end{namedtheorem}
\begin{rema}
  The relation between the Batyrev-Manin principle as stated here and
  the equidistribution may be described as follows: if the principle
  holds for a given open subset~$U$ for any metric on~$V$,
  then the points of $U(\KK)$ are equidistributed on~$V$. Conversely
  if the principle holds for a particular choice of the metric
  and an open subset~$U$ and if the
  points of $U(\KK)$ are equidistributed, then the principle
  holds for any choice of the metric (see \cite[\S3]{peyre:fano}).
\end{rema}

\subsubsection{The counterexample of Batyrev and Tschinkel}
This example was described in \cite{batyrevtschinkel:counter}.
We consider the hypersurface~$V$ in $\PP^3_\QQ\times\PP^3_\QQ$
defined by the equation
\[\sum_{i=0}^3X_iY_i^3=0.\]
We denote by $\mathcal O_V(a,b)$ the restriction
to~$V$ of the line bundle
$\pr_1^*(\mathcal O_{\PP^3_\QQ}(a))\otimes\pr_2^*(\mathcal O_{\PP^3_\QQ}(b))$
Then the anticanonical line bundle on $V$ is given
by $\mathcal O_V(3,1)$ and therefore the function $H:V(\QQ)\to\RR$
defined by
\[H(P,Q)=H_{\mathcal O_{\PP^3_\QQ}(1)}(P)^3H_{\mathcal O_{\PP^3_\QQ}(1)}(Q)\]
defines a height relative to the anticanonical line bundle
on~$V$. Let $\pi$ be the projection on the first factor and for any
$P\in\PP^3(\QQ)$, let $V_P=\pi^{-1}(P)$ the fibre over~$P$.
If $P=[x_0:x_1:x_2:x_2]$ with $\prod_{i=0}^3x_i\neq 0$, then
the fibre $V_P$ is a smooth cubic surface which contains $27$
projective lines.
The complement $U_P$ of these $27$ lines is defined over~$\QQ$.
For cubic surfaces, it is expected that the Batyrev-Manin principle holds
for any dense open subset contained in~$U_P$.
For any $P$ as above, let $t_P=\rk(\Pic(V_P))$ be the rank of the Picard
group of the cubic surface corresponding to~$P$.
Thus, according to \eqref{equ:batyrevmanin},
one expects that for any $U\subset U_P$, one has
\[\card U(\QQ)_{H\leq B}\sim C(V_P)B\log(B)^{t_P-1}\]
as $B$ goes to infinity.
One can show
that $t_P\in\{1,2,3,4\}$ and that $t_P=4$ if all the quotients
$x_i/x_j$ are cubes, that is if $P$ is in the image
of the morphism $c$ from $\PP_\QQ^3$ to $\PP_\QQ^3$ defined by
$[x_0:x_1:x_2:x_3]\mapsto[x_0^3:x_1^3:x_2^3:x_3^3]$.
But, on the other hand, by Lefschetz theorem,
the application $(a,b)\mapsto \mathcal O_V(a,b)$ induces
an isomorphims of groups from $\ZZ^2$ to $\Pic(V)$.
Therefore, the principle of Batyrev and Manin would be satisfied
for~$V$ if and only if there existed an open subset~$U$ of~$V$
such that
\[\card U(\QQ)_{H\leq B}\sim C(V)B\log(B)\]
as $B$ goes to infinity. Since the rational points in the
image of~$c$ are dense for Zariski topology, the open set $U$ has
to intersect an open set $U_P$ for some~$P$. Thus the
principle can not hold for both the cubic surfaces and $V$ itself.
\begin{listrems}
  \remark
  In fact, V.~V. Batyrev and Y.~Tschinkel proved in
  \cite{batyrevtschinkel:counter}
  that any dense open set of $V$ contains too many rational
  points over $\QQ(j)$,
  where $j$ is a primitive third root of unity. More recently,
  C.~Frei, D. Loughran, and E.~Sofos proved in
  \cite{freiloughransofos:bundle} that it is in fact the case over
  any number field.
  \remark
  One may look at the set
  \[T=\{\,P\in\PP^3(\QQ)\mid\rk(\Pic(V_P))>1\,\}\]
  that is the set of points for which the rank of the Picard group
  is bigger than the generic one. As we are about to explain,
  \[\card T_{H\leq B}=o(\card \PP^3(\QQ)_{H\leq B})\]
  which means that most of the fibers have a Picard group of rank one.
\end{listrems}
This example lead to the introduction of a new kind of accumulating subsets,
namely thin subsets (see J.-P.~Serre
\cite[\S3.1]{serre:topicsgalois}).
\begin{defi}%
  \label{defi:thin}
  Let~$V$ be a nice variety over the number field~$\KK$.
  A subset $T\subset V(\KK)$ is said to be \emph{thin}\index{Thin subsets},
  if there exists a morphism of varieties $\varphi: X\to V$ which
  satisfies the following conditions:
  \begin{conditions}
  \item
    The morphism~$\varphi$ is generically finite;
  \item
    The morphism~$\varphi$ has no rational section;
  \item
    The set~$T$ is contained in the image of~$\varphi$.
  \end{conditions}
\end{defi}
\begin{listrems}
  \remark
  If~$E$ is an elliptic curve, the group $E(\KK)/2E(\KK)$ is a finite
  group. Let $(P_i)_{i\in I}$ be a finite family of points of $E(\KK)$
  containing a representant for each element of $E(\KK)/2E(\KK)$.
  Then the morphism $\varphi:\coprod_{i\in I}E\to E$ which maps
  a point $P$ in the $i$-th component to $P_i+2P$ gives
  a surjective map onto the sets of rational points. This shows
  that $E(\KK)$ itself is thin.
  \remark
  In the example of Batyrev and Tschinkel,
  as~$T$ is a thin subset in $\PP^3(\QQ)$, it follows from
  \cite[\S13,theorem 3]{serre:mordellweil} that
  \[\card T_{H\leq B}=o(\card\PP^3(\QQ)_{H\leq B}).\]
  The set
  \[V_T=\bigcup_{P\in T}V_P(\QQ)\]
  is itself a thin subset of $V(\QQ)$. Conjecturally we may hope that
  \[\card(V(\QQ)\setminus V_T)_{H\leq B}\sim C_H(V)B\log(B)\]
  as~$B$ goes to infinity.
  In other words, the points on the complement of the accumulating
  subset should behave as expected. We shall explain below how
  a result of this kind was proven by C. Le Rudulier for a Hilbert scheme
  of the projective plane \cite{rudulier:hilbert}.
  More recently, T.~Browning and D.R.~Heath-Brown
  \cite{browningheathbrown:quadric} proved that
  for the hypersurface of $\PP^3_\QQ\times\PP^3_\QQ$ defined
  by the equation
  \[\sum_{i=0}^3X_iY_i^2\]
  the number of points on the complement of an accumulating
  thin subset behaves as expected.
  \remark
  The work of B.~Lehmann, S.~Tanimoto and
  Y.~Tschinkel~\cite{lehmanntanimototschinkel} shows how common varieties with
  accumulating thin subsets probably are.
  \remark
  We may assume that $\varphi$ is a proper morphism. Then
  $\varphi (X(\Adeles_\KK))\subset V(\Adeles_\KK)$ is a closed subset.
  Under mild hypotheses, T.~Browning and D.~Loughran proved
  in \cite{browningloughran:toomany}
  that
  \[\mmu_V(\varphi(X(\Adeles_\KK)))=0.\]
  Thus the existence of such a thin subset with a positive contribution
  to the asymptotic number of points is an obstruction to the
  global equidistribution of points.
\end{listrems}
\subsubsection{The example of C.~Le Rudulier}
C.~Le Rudulier considers the Hilbert scheme $V$ which parametrizes
the points of degree~$2$ in $\PP^2_\QQ$. To describe this scheme,
let us consider the sheme $Y$ defined as the second symmetric
product of $\PP^2_\QQ$:
\[Y=\Sym^2(\PP^2_\QQ)=(\PP^2_\QQ)^2/\mathfrak S_2.\]
More precisely, we may define it as the projective scheme
associated to the ring of invariant
polynomials $\QQ[X_1,Y_1,Z_1,X_2,Y_2,Z_2]^{\mathfrak S_2}$.
Let us denote by $\Delta_Y$ the image of the diagonal~$\Delta$ in~$Y$.
The scheme $Y$ is singular along this diagonal and $V$ may be
seen as the blowing up of $Y$ along the diagonal $\Delta_Y$.
From this point of view, the variety~$V$ is a desingularization
of~$Y$. Let define $P$ as the blowing up of $(\PP^2_\QQ)^2$ along
the diagonal.
We get a cartesian square
\[\xymatrix@-1.75pc{P\ar[rr]\ar[dd]_{\tilde\pi}&&(\PP^2_\QQ)^2\ar[dd]^\pi\\
&\ \rlap{$\square$}&\\
V\ar[rr]_b&&Y}\]
We put $\Delta_V=b^{-1}(\Delta_Y)$ and $U_0=V\setminus\Delta_V$.
then the set $T=\tilde\pi(P(\QQ))\cap U_0(\QQ)$ is a Zariski dense
thin accumulating subset. More precisely, C.~Le Rudulier proves
the following theorem:
\begin{theo}[C.~Le Rudulier]
  \textup{a)}
  Asymptotically the points of~$T$ give a positive contribution to the
  total number of points:
  \[\CDat\frac{\card T_{H\leq B}}{\card U_0(\QQ)_{H\leq B}}
  @>>B\to+\infty>c\]
  for a real number $c>0$. But for any strictly closed subset
  $F\subset V$, one has
  \[\card(F(\QQ)\cap T)_{H\leq B}=o(U_0(\QQ)_{H\leq B}).\]
  \par
  \textup{b)}
  On the complement of~$T$, one has
  \[\card(U_0(\QQ)\setminus T)_{H\leq B}\sim C(V)B\log(B)\]
  as $B\to +\infty$.
\end{theo}
\begin{listrems}
  \remark
  It follows from this theorem that the set $T$ is a thin
  subset which is not the union of accumulating subvarieties
  but which gives a positive contribution to the total number of points
  of bounded height on the variety.
  In the adelic space the closure of the points
  of $T$ are contained in a closed subset~$F$ with a volume
  $\mmu_V(F)=0$. Therefore this thin accumulating subset is an obstruction
  to the equidistribution of the points on~$V$.
  \remark
  Hopefully, in general, if $\omega_V^{-1}$ is ``big enough'',
  there should be a natural ``small'' subset $T$ such that the points of
  bounded height on $W=V(\KK)\setminus T$ should behave as expected.
  The problem is to describe this subset~$T$.
  \remark
  In these notes, so far, we did not go into the distribution
  of the rational points of bounded height for a height associated
  to an ample line with a class which is not a multiple of $\omega_V^{-1}$.
  The description in that case requires to introduce more complicated
  measures and we refer the interested reader to the work of V.~V. batyrev
  and Y. Tschinkel (see \cite{batyrevtschinkel:tamagawa}).
\end{listrems}
\section{All the heights}\label{section:allheights}
\subsection{Heights systems}
A natural approach to select the points we wish to keep is
to introduce more invariants. The rest of this lecture is devoted
to such invariants. Let us start by considering other heights.
Traditionally, most authors in arithmetic geometry consider
only one height given by a given ample line bundle. However
there are no reason to do so, and we may consider the whole
information given by heights. In order to do this, let
us introduce the notion of family of heights.
\begin{defi}
  Let $L$ and $L'$ be adelically normed line bundles
  on a nice variety~$X$. Let $(\Vert\cdot\Vert_w)_{w\in\Val(\KK)}$
  be the adelic norm on~$L$.
  We say that~$L$ and~$L'$ are \emph{equivalent}%
  \index{Equivalence of>adelically normed line bundle}
  if there is an integer $M>0$, a family $(\lambda_w)_{w\in\Val(\KK)}$ in
  $\RR_{>0}^{(\Val(\KK))}$, such that its support
  $\{\,w\in\Val(\KK)\mid \lambda_w\neq 1\,\}$ is finite
  and $\prod_{w\in\Val(\KK)}\lambda_w=1$, and an isomorphism
  from the line bundle $L^{\otimes M}$ equipped with the adelic
  norm $(\lambda_w{\Vert\cdot\Vert}^{\otimes M}_w)_{w\in\Val(\KK)}$ to
  the adelically normed line bundle ${L'}^{\otimes M}$. We denote by
  $\mathcal H(X)$ the set of equivalence classes of adelically
  normed line bundles. It has a structure of group induced
  by the tensor product of line bundles, we call this group
  the group of \emph{Arakelov heights on $X$}.%
  \index{Arakelov height}%
  \index{Height>Arakelov=(Arakelov ---)}%
  \glossary{$\mathcal H(X)$: Arakelov heights.}
\end{defi}
\begin{rema}
  The height introduced in definition~\ref{defi:arakelovheight}
  depends only on the equivalence class of the adelically normed
  line bundle $\det(E)$. From that point of view, the group
  $\mathcal H(X)$ does parametrize the heights on~$X$.
  If $X$ satisfies weak approximation and has an adelic
  point, then two distinct elements
  of $\mathcal H(X)$ define heights which differ at least at one rational
  point.
\end{rema}
\begin{exam}
  If $V$ is the spectrum of a point, then the height defines an isomorphism
  from $\mathcal H(V)$ to $\RR_{>0}$. Indeed, it is surjective
  and if we take a representative~$L$
  of an element of $\mathcal H(\Spec(\KK))$ of height~$1$, then let~$y$
  be an nonzero element of~$L$. The unique morphism of vector spaces
  from~$\KK$ to~$L$ which maps~$1$ to~$y$ then induces
  an isomorphism from~$\KK$ equipped with the adelic norm
  $(\Vert y\Vert_w\,|\cdot|_w)_{w\in\Val(\KK)}$ to~$L$.
\end{exam}
\begin{defi}
  \label{defi:systemofheights}
  A \emph{system of Arakelov heights}%
  \index{System of Arakelov heights} on our nice variety~$V$
  is a section $\boldsymbol s$
  of the forgetful morphism of groups
  \[\boldsymbol o:\mathcal H(V)\longrightarrow \Pic(V).\]
  Such a system defines a map
  \[\multih:V(\KK)\to\Pic(V)_\RR\dual\]
  constructed as follows: for any $P\in V(\KK)$ and any $L\in\Pic(V)$,
  the real number
  $\multih(P)(L)$ is the logarithmic height of the point~$P$
  relative to the Arakelov height $\boldsymbol s(L)$.
  We shall call $\multih(P)$ the \emph{multiheight}\index{Multiheight}
  of the point~$P$.
  By abuse of language, a function
  of the form $P\mapsto \exp(\langle u,\multih(P)\rangle)$
  for some $u\in\Pic(V)_\RR$ will also be called an exponential height
  on~$V$.
\end{defi}
Let us fix a system of Arakelov heights on our nice variety~$V$.
We still assume that~$V$
satisfies the hypotheses~\ref{hypos:conditions}.
Then one can study the multiheights of rational points.
\begin{lemm}
  Under the hypotheses~\ref{hypos:conditions},
  there is a dense open subset~$U$ of~$V$ and an element
  $c\in\Pic(V)_\RR\dual$ such
  that
  \[\forall P\in U(\KK),\quad\multih(P)\in c+\Ceffdof V.\]
\end{lemm}
\begin{proof}
  Let $L_1,\dots,L_m$ be line bundles the classes of which
  generate the effective cone in $\Pic(V)_\RR$.
  We may assume that they have nonzero sections.
  Let $U$ be the complement of the base loci of these line
  bundles. Let $i\in\{1,\dots,m\}$.
  Then choosing a basis $(s_0,\dots,s_{N_i})$ of the space
  of sections of the line bundle $L_i$, we get a morphism
  from $U$ to a projective space $\PP_\KK^{N_i}$.
  For any place $w$, there exist a constant $c_w$ such that
  $\Vert s_j(x)\Vert_w\leq c_w$ for any $x\in V(\KK_w)$ and any
  $j\in\{0,\dots,N_i\}$. Moreover we may take $c_w=1$ outside a finite
  set of places. Therefore there exists a constant $C$ such that
  for any $x\in U(\KK)$ there is an $j\in\{0,\dots,N_i\}$ with
  \[0<\prod_{w\in\Val(\KK)}\Vert s_j(x)\Vert\leq C.\]
  It follows that there exists a constant $c_i\in\RR$ such that
  $h_i(P)\geq c_i$ for any $P\in U(\KK)$. The statement of
  the lemma follows.
\end{proof}
\begin{rema}
  Let $\Ceffdoof{V}$ be the interior
  of the dual cone $\Ceffdof{V}$.
  This lemma shows that it is quite natural to count the number of rational
  points in $V(K)$ such that $\multih(P)\in \mathcal D_B$
  for some compact domain $\mathcal D_B\subset\Ceffdoof{V}$ depending on
  a parameter $B\in\RRpp$. In the following,
  we shall consider domains of the form
  \[\mathcal D_B=\mathcal D_1+\log(B)u\]
  where $u\in\Ceffdoof{V}$ and $\mathcal D_1$ is a compact
  polyhedron in $\Pic(V)_\RR\dual$. In other words,
  we get a finite number of conditions of
  the form
  \[aB\leq H(P)\leq bB\]
  where $H$ is an exponential height on~$V$, in the sense
  of definition~\ref{defi:systemofheights}, and $a,b\in\RRpp$.
\end{rema}
\begin{nota}
  We define the measure $\nu$ on $\Pic(V)_\RR\dual$ as follows:
  for a compact subset $\mathcal D$ of $\Pic(V)_\RR\dual$,
  \[\nu(\mathcal D)=\int_{\mathcal D}e^{\langle\omega_V^{-1},
    \boldsymbol y\rangle}\Haar {\boldsymbol y},\]
  where the Haar measure $\Haar {\boldsymbol y}$
  on $\Pic(V)_\RR\dual$ is normalised so that
  the covolume of the dual of the Picard group is one.
  \par
  For any domain $\mathcal D\subset\Pic(V)_\RR\dual$, we define
  \[V(\KK)_{\multih\in\mathcal D}=\{\,P\in V(\KK)\mid
  \multih(P)\in \mathcal D\,\}\]
\end{nota}
With these notations, we may ask the following question:
\begin{question}\label{question:allheights}
  We assume that our nice variety~$V$ satisfies the conditions
  of the hypothesis~\ref{hypos:conditions}. Let $\mathcal D_1$
  be a compact polyhedron of $\Pic(V)_\RR\dual$ and $u$
  be an element of the open cone $\Ceffdoof{V}$. For a real number $B> 1$,
  let $\mathcal D_B=\mathcal D_1+\log(B)u$.
  Can we find a ``small'' subset~$T$ so that we have an equivalence of the form
  \begin{equation}
    \label{equ:estimateallheights}
    \card (V(\KK)-T)_{h\in\mathcal D_B}\sim \beta(V)\nu(\mathcal D_1)
      \oomega_V(V(\Adeles_\KK)^{\Br})B^{\langle\omega_V^{-1},u\rangle}
  \end{equation}
  as $B$ goes to infinity?
\end{question}
\begin{listrems}
  \remark\label{remark:db}
  One may note that in the right hand side of~\eqref{equ:estimateallheights},
  one may use $\nu(\mathcal D_B)$ =
  $\nu(\mathcal D_1)B^{\langle\omega_V^{-1},u\rangle}$.
  \remark
  One can easily imagine variants of this question. For example,
  some methods from analytic number theory give much better error
  terms if ones use smooth functions instead of characteristic functions
  of sets. So it would be natural to consider a smooth
  function $\varphi:\Pic(V)\dual_\RR\to\RR$
  with compact support and ask whether we have
  \[\sum_{P\in V(\KK)}\varphi(\multih(P)-Bu)\sim
  \beta(V)\int_{\Pic(V)_\RR\dual}\varphi
  \Haar{\nu}\oomega_V(V(\Adeles_\KK)^{\Br})B^{\langle\omega_V^{-1},u\rangle}\]
  as $B$ goes to infinity.
  \remark
  Let us compare formula~\eqref{equ:estimateallheights} with
  formula~\eqref{equ:batyrevmanin}. First we may note that
  \[\nu(\{y\in\Ceffdof V\mid\langle y,\omega_V^{-1}\rangle\leq B\})
  \sim \alpha(V)B\log(B)^{t-1}\]
  Thus using remark~\ref{remark:db},
  formula~\eqref{equ:batyrevmanin} may be seen as integrating
  formula~\eqref{equ:estimateallheights} over
  \[\mathcal D_B=\{y\in\Ceffdof V\mid\langle y,\omega_V^{-1}\rangle\leq 1\}.\]
\end{listrems}
In this context in which we consider all the possible
heights, we may consider again the question of the
global equidistribution.
\begin{namedtheorem}[Global equidistribution]%
  \label{globalequidistribution:allheights}
  We shall say that the \emph{global equidistribution holds for $\multih$}%
  \index{Global equidistribution>for systems of heights}%
  \index{Equidistribution>Global=(Global ---)}
  if, for any compact polyhedron~$\mathcal D$ in
  $\Pic(V)_\RR\dual$ and any~$u$ in the open cone $\Ceffdoof{V}$,
  the measure $\ddelta_{V(\KK)_{\multih\in\mathcal D_B}}$ converges
  weakly to $\mmu_V^{\Br}$ as $B$ goes to infinity.
\end{namedtheorem}
Note that the expected limit probability measure is the
same as before and does not depend on~$u$.

\subsection{Compatibility with the product}
A positive answer to question~\ref{question:allheights} is
compatible with the product of varieties in the following sense:
\begin{prop}
  Let $V_1$ and $V_2$ be nice varieties equipped with system of heights
  which satisfy the
  conditions~\ref{hypos:conditions}.
  If the sets $V_1(\KK)-T_1$ and $V_2(\KK)-T_2$ satisfy the
  equivalences~\eqref{equ:estimateallheights} for any compact polyhedra,
  then this is also true for the
  product~$(V_1(\KK)-T_1)\times(V_2(\KK)-T_2)$, equipped
  with the induced system of heights.
  \par
  If these varieties satisfy the global
  equidistribution~\ref{globalequidistribution:allheights},
  then so does their product.
\end{prop}
\begin{proof}
  We put $W_i=V_i(\KK)-T_i$ for $i\in\{1,2\}$.
  Let~$W$ be the product $W_1\times W_2$. For $i\in\{1,2\}$,
  we denote by $\multih_i$ the multiheight on~$V_i$, and fix
  a compact polyhedron $\mathcal D_{i,1}$ in $\Pic(V_i)_\RR\dual$,
  as well as an element $u_i\in\Ceffdoof{V_i}$.
  Let us first note that by \cite[exercise 12.6]{hartshorne:geometry},
  the natural morphism induced by pull-backs $\Pic(V_1)\times\Pic(V_2)
  \to\Pic(V)$ is an isomorphism which maps
  the product $\Ceffof{V_1}\times\Ceffof{V_2}$
  onto $\Ceffof{V}$ and $(\omega_{V_1}^{-1},\omega_{V_2}^{-1})$
  on $\omega_V^{-1}$ (see \cite[exercise 8.3]{hartshorne:geometry}).
  Therefore we identify
  these groups and consider $\mathcal D_1=\mathcal D_{1,1}\times
  \mathcal D_{2,1}$
  as a subset of $\Pic(V)_\RR\dual$ and $u=(u_1,u_2)$ as
  an element of $\Ceffdoof{V}$. If we put
  $\mathcal D_B=\log(B)u+\mathcal D_1$,
  we have
  \[\card W_{\multih\in\mathcal D_B}=
  \card (W_1)_{\multih_1\in\mathcal D_{1,B}}
  \times\card (W_2)_{\multih_2\in\mathcal D_{2,B}}\]
  and the result follows from the compatibility of equivalence
  with products. To extend the result to an arbitrary polyhedra $\mathcal D$,
  we find domains $\mathcal D'$ and $\mathcal D''$ which are finite unions
  of products of polyhedra with disjoint interiors such that
  $\mathcal D'\subset\mathcal D\subset\mathcal D''$ and use the fact that
  the equivalence is valid for such a finite union.

  Similarly for the equidistribution, it is enough to count
  the points in open subsets $U$ of $V(\Adeles_\KK)^{\Br}$ which are
  of the form $U=U_1\times U_2$ for open subsets $U_1$ and $U_2$
  such that $\oomega_{V_1}(\partial U_1)=0$ and $\oomega_{V_2}(\partial U_2)=0$.
  But in that case,
  \[\card(W\cap U)_{\multih\in\mathcal D_B}=
  \card (W_1\cap U_1)_{\multih_1\in\mathcal D_{1,B}}
  \times\card (W_2\cap U_2)_{\multih_2\in\mathcal D_{2,B}}\]
  and we may conclude in the same way.
\end{proof}
It is worthwile to note that this proof is much simpler than the
proof of the compatibility of the principle of Batyrev and Manin for
products (see \cite[\S1.1]{fmt:fano}). It illustrates the fact that in
question~\ref{equ:estimateallheights} we cut out the ``spikes''
where the heights of the components of the points are very different.

\subsection{Lifting to versal torsors}
Following Salberger~\cite{salberger:tamagawa}, we shall
now explain how the question lifts naturally to versal torsors
(see also \cite{peyre:cercle}). Let us start by a quick reminder
on versal torsors. In our setting, the geometric Picard
is supposed to be without torsion, thus we shall
restrict ourselves to torsors under algebraic tori.
\begin{defi}
  Let~$\LL$ be a field and $\LL^s$ be a separable closure
  of~$\LL$. For any scheme~$X$ over~$\LL$, we write
  $X^s$ for the product $X\times_{\Spec(\LL)}\Spec(\LL^s)$.%
  \glossary{$X^s$: Extension of scalars to the separable closure}

  An algebraic group~$G$ over a field~$\LL$ is said to be
  \emph{of multiplicative type}%
  \index{Group>of multiplicative type}
  if there exists an integer~$n$ such that $G^s$ is isomorphic
  to a closed subgroup of $\GG_{m,\LL^s}^n$.
  A \emph{torus}~$\torus$\index{Torus} over~$\LL$
  is an algebraic group~$\torus$
  over~$\LL$ such that~$\torus^s$ is isomorphic to a power $\GG_{m,\LL^s}^n$
  of the multiplicative group.

  The \emph{group of characters}%
  \index{Characters} of an algebraic group~$G$, denoted by $X^*(G_s)$%
  \glossary{$X^*(H_s)$: group of characters}
  is the group of group homomorphisms from~$G_s$
  to $\GG_{m,\LL^s}$. If~$G$ is of multiplicative type,
  it is a finitely generated $\ZZ$-module. If~$G$ is a torus,
  it is a free $\ZZ$-module of rank $n$. In both cases, it is equipped
  with an action of the absolute Galois group of $\LL$,
  that is $\mathcal G_\LL=\Gal(\LL^s/\LL)$, which splits
  over a finite separable extension of $\LL$.

  Conversely, a \emph{Galois module~$L$ over~$\LL$}
  (resp. a \emph{Galois lattice~$L$ over~$\LL$})%
  \index{Galois lattice}%
  \index{Lattice>Galois=(Galois ---)} is
  a finitely generated $\ZZ$-module (resp. a free $\ZZ$-module
  of finite rank) equipped with an action
  of the Galois group $\mathcal G_\LL$ which splits over a finite
  extension. To a Galois module~$L$, we may associate
  the monoid algebra $\LL^s[L]$ and thus
  the algebraic variety
  \[\torus=\Spec(\LL^s[L]^{\mathcal G_\LL})\]
  equipped with the algebraic group structure
  induced by the coproduct $\nabla$ on $\LL^s[L]$ d\'efined by
  $\nabla(\lambda)=\lambda\otimes\lambda$ for any $\lambda\in L$.
  This algebraic group is an algebraic group
  of multiplicative type, which we shall say to be
  \emph{associated to~$L$}.
\end{defi}
\begin{exam}
  As a basic example, the group of characters of $\GG_{m,\LL}^n$
  is $\ZZ^n$ with a trivial action of the Galois group
  and the torus associated with $\ZZ^n$ is isomorphic
  to $\GG_{m,\LL}^n$.
\end{exam}
\begin{rema}
  These construction are functorial and we get a contravariant equivalence
  of categories between the category of tori (resp. groups
  of multiplicative type) over~$\LL$ and the category
  of Galois lattices (resp. Galois modules) over~$\LL$.
\end{rema}
\begin{nota}
  We shall denote by $\TNS$ the torus associated
  to the Galois lattice $\Pic(\overline V)$.
\end{nota}
We are going to use pointed torsors, that is torsors
in the category of pointed schemes.
\begin{defi}
  Let~$G$ be an algebraic group over a field~$\LL$ and let~$X$
  be an algebraic variety over~$\LL$. A \emph{$G$-torsor}~$\torsor$
  over $X$\index{Torsor} is an algebraic variety~$T$ over~$\LL$
  equipped with a faithfully flat morphism
  $\pi:\torsor\to X$ and an action $\mu:G\times \torsor\to \torsor$ of~$G$
  such that $\pi\circ\mu=\pi\circ\pr_2$ and
  the morphism given by $(g,y)\mapsto(gy,y)$ is an isomorphism
  from $G\times \torsor$ to $\torsor\times_X\torsor$.

  A \emph{pointed variety}\index{Pointed>variety}
  over~$\LL$ is a variety~$X$ over~$\LL$
  equipped with a chosen rational point $x\in X(\LL)$
  A \emph{pointed torsor}\index{Pointed>torsor}%
  \index{Torsor>Pointed=(Pointed ---)} over the pointed variety~$X$
  is a torsor~$\torsor$ over~$X$
  equipped with a rational point $t\in \torsor(\LL)$ such that $\pi(t)=x$.
\end{defi}
\begin{exam}
  For any line bundle~$L$ over~$X$, we can define a $\GG_{m,\LL}$ torsor
  by considering~$L^\times$ which is the complement of the zero section
  in~$L$. Conversely for a nice variety~$X$, given a $\GG_m$ torsor~$\torsor$,
  we get a line bundle by considering the contracted product
  $\torsor\times^{\GG_{m,\LL}}\Affine^1_\LL$ which is the quotient
  $(\torsor\times \Affine^1_\LL)/\GG_{m,\LL}$
  where $\GG_{m,\LL}$ acts by $t.(y,a)=(t.y,t^{-1}.a)$.
  We get in that way the equivalence of category between
  the line bundles and the $\GG_{m,\LL}$-torsors over~$X$.
\end{exam}

\subsubsection{Versal and universal torsors}
The versal torsors were introduced by J.-L. Colliot-Th\'el\`ene
and J.-J. Sansuc in the study of the Brauer-Manin obstruction
for Hasse principle and weak approximation (see \cite{cts:predescente1},
\cite{cts:descente1}, and \cite{cts:descente2}) .
For a survey on versal torsors, the reader may also look
at~\cite{peyre:obstruction}.
\par
In topology, universal coverings for an unlaceable pointed space~$X$
answers a universal problem for coverings: it is a
pointed covering $E\to X$ such that for any pointed covering $C\to X$
there exists a unique morphism $E\to C$ of pointed spaces over~$X$
(see \cite{bourbaki:topoalg4}).
We could in fact restrict ourselves to Galois coverings, that is
connected coverings with an automorphism group which acts transitively
on the fibre over the marked point of~$X$. Fixing a point in the space~$X$
is necessary to guarantee the unicity, up to a unique isomorphism,
of the universal covering. The universal torsor is the answer to a similar
problem for torsor under groups of multiplicative type.
\begin{defi}
  Let~$\LL$ be a field and~$\overline \LL$ be an algebraic closure
  of~$\LL$. Let $X$ be a smooth and geometrically integral variety over
  a field~$L$ with a rational point such that all invertible
  functions on $\overline X$ are constant:
  $\Gamma(\overline X,\GG_m)=\overline L^*$.
  We see $X$ as a pointed space by fixing a rational point $x\in X(\LL)$.
  Then a \emph{universal torsor}\index{Universal torsor}%
  \index{Torsor>universal=(Universal ---)} is a pointed torsor~$\torsor_u$
  over the pointed space~$X$ under a group
  of multiplicative type $\torus_u$ such that for any pointed torsor~$\torsor$
  over~$X$ under a group of multiplicative type $\torus$,
  there is a unique morphism of group $\varphi:\torus_u\to\torus$ and a
  unique morphism $\psi:\torsor_u\to \torsor$ over~$X$, compatible
  with the actions of~$\torus_u$ and~$\torus$ and the marked points.
\end{defi}
\begin{listrems}
  \remark
  If such a torsor exists it is obviously unique up to a unique
  isomorphism.
  \remark
  The extension of scalars of a universal torsor is also
  a universal torsor.
  \remark\label{rema:torsor:closure}
  Let us assume that there exists a universal torsor $\torsor_u$.
  Let $x$ be the chosen point of $X$.
  For any line bundle $L$ over $X$, we can consider the
  $\GG_m$-torsor $L^\times$ and fix a point in its fibre
  over~$x$. Thus there exists a unique morphism of pointed
  torsors from $\torsor_u$ to $L^\times$ compatible with a morphism
  $\torus_u\to\GG_m$.
  By duality, it corresponds to a homomorphism of groups
  from $\ZZ$ to the group of characters of $\torus_u$.
  Moreover if $L^{\otimes n}$ is isomorphic to the trivial line bundle,
  the image of $n\in\ZZ$ in $X^*(\torus_u)$ is trivial.
  Therefore, over $\LL^s$, we get a homomorphism of groups
  from $\Pic(X^s)$ to $X^*(\torus_u^s)$, which is compatible with the
  Galois actions.

  Conversely, for any torsor~$\torsor$ under a multiplicative
  group~$\torus$ and any group
  character~${\chi:\torus\to\GG_m}$,
  the contracted product $\torsor\times^\torus\GG_{m,\LL}$
  is a $\GG_m$ torsor over $\LL$. We get a homomorphism of groups
  from $X^*(\torus^s)$ to $\Pic(X^s)$. It is possible to deduce
  from such arguments that the character group of $\torus_u$
  over $\LL^s$ has to be isomorphic to $\Pic(X^s)$.
\end{listrems}
\begin{cons}
  Let us now explain how it is possible to construct such universal torsors.
  We shall assume again hypothesis~\ref{hypos:conditions},
  and fix a rational point $x\in V(\KK)$. In that case the group $\torus_u$
  is canonically isomorphic to the N\'eron-Severi torus $\TNS$.
  Over $\overline \KK$, the construction of remark~\ref{rema:torsor:closure}
  gives
  an isomorphism of $\TNS$-torsor from a universal torsor
  $\overline\torsor_u$ to the product
  $L_1^\times\times_V\dots\times_VL_t^\times$
  where $([L_1],\dots,[L_t])$ is a basis of $\Pic(\overline V)$.
  But the unicity of the universal torsor shows that, by marking
  $\overline\torsor_u$
  with a point in the fibre of $x$, there exists no non-trivial automorphism
  of $\overline\torsor_u$ as a pointed torsor over $X$. By descent theory,
  $\overline\torsor_u$ comes from a unique pointed $\TNS$-torsor
  $\torsor_u$ over~$X$.
\end{cons}
\begin{rema}
  In particular, as a non-pointed $\TNSb$-torsor
  over~$\overline V$, the torsor $\overline\torsor_u$
  does not depend on the choice of the point $x$ in $V(\KK)$.
  This is not true over~$\KK$.
\end{rema}
\begin{defi}
  A \emph{versal torsor}\index{Versal torsor}%
  \index{Torsor>versal=(versal ---)} over~$V$ is a $\KK$-form of
  the $\TNSb$-torsor
  $\overline \torsor_u$.
\end{defi}
\begin{rema}
  The automorphisms of $\overline \torsor_u$ as a $\TNSb$-torsor
  over $\overline V$ are given by the action of $\TNS(\overline \KK)$.
  It follows that if we fix a rational point, and
  therefore a universal torsor $\torsor_u$,
  the versal torsors are classified by the
  group of Galois cohomology $H^1(\KK,\TNS)$ and we get a map
  from $V(\KK)$ to $H^1(\KK,\TNS)$ which maps a point to the
  class of the corresponding universal torsor. In general this cohomology
  group is infinite. But Colliot-Th\'el\`ene and Sansuc
  proved in \cite[proposition 2]{cts:descente1} that the image of the map
  is finite. In other words, there exists a finite family
  $(\torsor_i)_{i\in I}$ of non-isomorphic versal torsors over $V$
  with a rational point such that
  \[V(\KK)=\coprod_{i\in I}\pi_i(\torsor_i(\KK)),\]
  where $\pi_i:\torsor_i\to V$ is the structural morphism.
\end{rema}

\subsubsection{Structures on versal torsors}
Let~$\torsor_u$ be a universal torsor over~$V$.
By definition of the torsors, there is a natural isomorphism
\[\torsor_u\times_V \torsor_u\iso \TNS\times \torsor_u\]
which shows that the pull-back of $\torsor_u$ to $\torsor_u$ is trivial.
But from the universality of~$\torsor_u$
it is possible to show that the pull-back
of any pointed torsor under a group of multiplicative type is trivial
\cite[proposition 2.1.1]{cts:descente2}. By this proposition,
we also have that invertible functions on $\torsor_u$
are constant: $\Gamma(\torsor_u,\GG_m)=\KK^*$.
Moreover, by \cite[lemme 2.1.10]{peyre:cercle},
$\omega_{\torsor_u}$ is isomorphic to the pull-back of $\omega_V$.
We get the following assertion concerning \emph{volume forms}%
\index{Volume form},
that is non-vanishing sections of $\omega_{\torsor_u}$.
\begin{prop}
  Let~$\torsor$ be a versal torsor over~$V$. Then up to multiplication by a constant
  there exists a unique volume form on~$\torsor$.
\end{prop}
\begin{cons}
  \label{cons:measure.torsor}
  Let~$\torsor$ be a versal torsor on~$V$ with a rational point. 
  By the proposition, we may
  take a non-vanishing section $\omega$ of $\omega_{\torsor}$.
  For any place~$w$ of $\KK$, the expression
  \[\left|\left\langle\omega,\frac\partial{\partial x_1}\wedge
  \frac\partial{\partial x_2}\wedge\dots\wedge
  \frac\partial{\partial x_n}\right\rangle\right|_w
  \Haar{x_{1,w}}\Haar{x_{2,w}}\dots
  \Haar{x_{n,w}},\]
  defines a local measure, which, like in construction~\ref{cons:measure},
  we may patch together to get a measure
  $\oomega_{\torsor_u,w}$ on~$\torsor_u(\KK_w)$.

  We then choose a finite set~$S$ of places containing
  all the places of bad reduction for~$V$, the archimedean places,
  as well as the ramified places in a extension splitting the
  action of the Galois group on the Picard group of~$V$.
  Moreover, we may assume that any isomorphism class of versal torsors
  with a rational point has a model over the ring of~$S$-integers
  $\mathcal O_S$ and that the projection maps $\torsor(\KK_w)\to V(\KK_w)$
  are surjective for $w\not\in S$ (\cite[lemme 3.2.3]{cts:descente2}).
  Let us fix such a model $\mathcal \torsor$ of our
  versal torsor~$\torsor$. Then for any place $w$ outside a finite
  set of places,
  one can prove (see the proof of theorem~\ref{theo.lift.constant} below) that
  \[\oomega_{\torsor,w}(\mathcal \torsor(\mathcal O_w))=
  L_w(1,\Pic(\overline V))^{-1}
  \oomega_{V,w}(V(\KK_w)).\]
  Using the arguments of construction~\ref{cons:tamagawa},
  it follows that we can define the product of the measures
  \[\oomega_{\torsor}=\frac 1{\sqrt{d_\KK}^{\,\dim{\torsor}}}
  \prod_{w\in\Val(\KK)}\oomega_{\torsor,w}.\]
  on the adelic space $\torsor(\Adeles_\KK)$. By the product formula,
  this measure does not change if we multiply $\omega$ by a nonzero
  constant. Thus we may call $\oomega_\torsor$ the \emph{canonical measure}%
  \index{Canonical measure>on versal torsors}
  on the adelic space of the versal torsor $\torsor$.
\end{cons}
\begin{exam}
  For a smooth hypersurface~$V$ of degree~$d$ in $\PP^N_\KK$, with $N\geq 4$,
  any versal torsor is isomorphic to the cone over the hypersurface
  in $\Affine^{N+1}_\KK\setminus\{0\}$, and the canonical measure
  is given by the Leray form. If $F$ is a homogeneous equation for~$V$,
  then locally the measure may be defined as
  \[\oomega_{\torsor,w}=\frac 1{|F(1,x_1,\dots,x_N)|_w}
  \Haar{x_{1,w}}\dots\Haar{x_{N,w}}.\]
\end{exam}
Let us now turn to the lifting of heights to versal torsors.
We have to take into account that the rank of the Picard group
at a place~$w$ depends on~$w$.
\begin{cons}
  We choose a system of representants $(\torsor_i)_{i\in I}$
  of the isomorphism classes of versal torsors over~$V$ which have
  a rational points over $\KK$. For each $i\in I$, we also fix
  a point $y_i\in \torsor_i(\KK)$.
  Let~$\LL$ be a Galois extension of~$\KK$ which splits the Picard
  group of~$V$. Let $\boldsymbol s_\LL:\Pic(V_\LL)\to\mathcal H(V_\LL)$
  be a system of heights over~$\LL$.  We also fix
  a place $w_0$ of $\KK$. Let $i\in I$.
  For any line bundle $L$ over $\Pic(V_\LL)$ there exists
  a morphism $\phi_L:\torsor_i\to L^\times$ over~$V$, which is compatible
  with the character $\chi_L:(\TNS)_\LL\to\GG_{m,\LL}$ defined by $L$.
  This morphism is unique up to multiplication by a constant.
  Let us choose a representant $(\Vert\cdot\Vert_v)_{v\in\Val(\LL)}$
  of $s_\LL([L])$ defining the exponential height $H_\LL$ on $V(\LL)$.
  For any $v\in\Val(\LL)$, we may then consider the map
  from $\torsor_i(\LL_v)$ to $\RR$ given by
  \[y\longmapsto \Vert y\Vert^L_v=
  \begin{cases}
    \frac{\Vert\phi_L(y)\Vert_v}{\Vert\phi_L(y_i)\Vert_v}
    \text{ if }{v\mathrel{\not{\mid}} w_0}\\
    \frac{\Vert\phi_L(y)\Vert_v}{\Vert\phi_L(y_i)\Vert_v}
    H_L(\pi_i(y_i))^{-\frac{[\LL_v:\KK_w]}{[\LL:\KK]}}\text{ otherwise.}
  \end{cases}\]
  This map does not depend on the choice of $\phi_L$ nor on the choice
  of the representant of $s_\LL([L])$ and satisfies
  \[\forall y\in\torsor_i(\LL),\quad H_L(\pi_i(y))=\prod_{v\in\Val(\LL)}
  (\Vert y\Vert_v^L)^{-1}.\]
  Moreover it satisfies the formula
  $\Vert t.y\Vert^L_v=|\chi_L(t)|_v\Vert y\Vert^\LL_v$, for $t\in\TNS(\LL_v)$
  and $y$ in $\torsor_i(\LL_v)$.
  We get a map
  \[\liftmultih_v:\torsor_i(\LL_v)\longrightarrow(\Pic(V_{\LL_v}))\dual_\RR\]
  defined by the relations
  \[\Vert y\Vert_v^L=q_v^{-\bigl\langle\liftmultih_v(y),[L]\bigr\rangle}\]
  for $y\in\torsor_i(\LL_v)$ and $[L]\in\Pic(V_{\LL_v})$, with
  $q_v$ the cardinal of the residue field $\FF_v$ if $v$ is ultrametric,
  $q_v=e$ for a real place and $q_v=e^2$ for a complex one.
  Let us now write $V_w$ for $V_{\KK_w}$. Using the inclusion
  $\torsor_i(\KK_w)\to\prod_{v|w}\torsor_i(\LL_v)$ and the projection
  $\pr:\prod_{v|w}\Pic(V_{L_v})\dual_\RR\to\Pic(V_w)\dual_\RR$,
  we define a map
  \[\liftmultih_w:\torsor_i(\KK_w)\longrightarrow\Pic(V_w)_\RR\dual.\]
  so that the diagram
  \begin{equation}\label{equ.diag.norm}
    \vcenter{
      \hbox{
        \xymatrix{
          \prod_{v|w}\torsor_i(\LL_v)\ar[r]&\prod_{v|w}\Pic(V_{L_v})_\RR\dual
          \ar[d]^{\frac 1{[\LL:\KK]}\pr}\\
          \torsor_i(\KK_w)\ar[u]\ar[r]^{\liftmultih_w}&\Pic(V_{K_w})_\RR\dual
        }
      }
    }
  \end{equation}
  commutes.
  \par
  If $L$ is line bundle over $X$ and if $(\Vert\cdot\Vert_v)_{v\in\Val(\LL)}$
  is an adelic norm for the extension of scalars $L_\LL$, then it induces
  an adelic norm on $L$ defined by
  \[\forall w\in\Val(\KK),\forall y\in L(\KK_w),
  \Vert y\Vert_w=\left(\prod_{v|w}\Vert y\Vert_v
  \right)^{\rlap{$\scriptstyle\frac 1{[\LL:\KK]}$}}.\]
  Therefore the system of heights $\boldsymbol s_\LL$ induces
  a system of heights $\boldsymbol s:\Pic(V)\to\mathcal H(V)$.
  For any point $y\in\torsor_i(\KK)$ we have the formula
  \[\multih(\pi_i(y))=\sum_{w\in\Val(\KK)}\log(q_w)\liftmultih_w(y).\]
\end{cons}
These construction enables us to lift a
system of heights to versal torsors
with a rational point.

\subsubsection{Lifting of the asymptotic formula}
We now wish to express the asymptotic
formula~\eqref{equ:estimateallheights} at the torsor level.
The fibre of the projection map
$\pi_i:\torsor_i(\KK)\to V(\KK)$ is either empty or a principal homogeneous
space under $\TNS(\KK)$. Therefore
we now need to use the description of the rational points of the torus
$\TNS$, as described in the work of Ono (\cite{ono:algebraic} and
\cite{ono:tori}).
\begin{defi}
  Let $\torus$ be an algebraic torus over~$\KK$. We denote
  by $W(\torus)$ the torsion subgroup of $\torus(\KK)$.
  By an abuse of notation,
  for any place~$w$ of $\KK$, we denote by
  $\torus(\mathcal O_w)$ the maximal compact subgroup of
  $\torus(\KK_w)$.
  Let us put $K_\torus=\prod_{w\in\Val(\KK)}\torus(\mathcal O_w)$ which
  is a compact subgroup of $\torus(\Adeles_\KK)$. We also have
  that $W(\torus)=K_\torus\cap \torus(\KK)$
  For any place $w$,
  there is an injective morphism of groups
  \[\log_w:\torus(\KK_w)/\torus(\mathcal O_w)
  \longrightarrow X^*(\torus_w)_\RR\dual\]
  so that for any $t\in \torus(\KK_w)$ and any $\chi\in X^*(\torus_w)$,
  we have $q_w^{\langle\log_w(t),\chi\rangle}=|\chi(t)|_w$.
  For almost all places $w$ the image of $\log_w$ coincide with
  $X^*(\torus_w)\dual$. In fact, by \cite[theorem 4]{ono:algebraic} and
  \cite[\S3]{ono:tori} there exists a finite set of places $S_T$ such that
  the induced map gives an exact sequence
  \begin{equation}
    \label{equ.exact.ono}
    1\longrightarrow \torus(\mathcal O_{S_T})\longrightarrow
    \torus(\KK)\longrightarrow \bigoplus_{w\in\Val(\KK)-{S_T}}X^*(\torus)_w\dual
    \longrightarrow 0
  \end{equation}
  and there is an exact sequence
  \begin{equation}
    \label{equ.exact.ono2}
    {\CDat 1\longrightarrow W(\torus)\longrightarrow \torus(\mathcal O_{S_T})
      @>>\log_{S_T}>\bigoplus_{w\in S_T}X^*(\torus_w)_\RR\dual},
  \end{equation}
  where $\log_{S_T}$ is the map defined by taking $\log_w$
  for $w\in S_T$. For any $w\in S_T$, the extension of scalars
  defines a linear map $\pi_w:X^*(T_w)_\RR\dual\to X^*(T)_\RR\dual$.
  We then consider the linear map $\pi=\sum_{w\in S_T}\log(q_w)\pi_w$:
  \[\bigoplus_{w\in S_\torus}X^*(\torus_w)_\RR\dual
  \longrightarrow X^*(\torus)_\RR\dual.\]
  By the product formula, the image of $\torus(\mathcal O_{S_T})$
  is contained in $\ker(\pi)$. The image $M=\pi(\torus(\mathcal O_{S_\torus}))$
  is a lattice in the $\RR$-vector space $\ker(\pi)$. Let
  $(e_1,\dots,e_m)$ be a basis for this lattice and let
  \[\Delta=\Bigl\{\,\sum_{i=1}^mt_ie_i,\ (t_i)_{1\leq i\leq m}\in
  \leftclose0,1\rightopen^m\,\Bigr\}.\]
  By construction, $\Delta$ is a fundamental domain for the action of
  $\torus(\mathcal O_{S_\torus})$ on $\ker(\pi)$.
\end{defi}

\begin{cons}
  By increasing the finite set of places $S$ introduced in
  construction~\ref{cons:measure},
  we assume that we may take $S_{\TNS}=S$ for the finite set of places
  considered
  in the last definition.
  In particular,
  we get that outside of $S$, the map
  \[\TNS(\KK_w)\longrightarrow \Pic(V_w)\dual\]
  is surjective.
  For each of the chosen torsors we may also fix models $\mathcal \torsor_i$
  over $\mathcal O_S$. 
  We may further assume that, for a family of line bundles
  which generates $\Pic(V_\LL)$ and is invariant under the action
  of the Galois group $\Gal(\LL/\KK)$,
  the heights are given by models of the corresponding line bundles and
  that the maps $\phi_L$ from the chosen versal torsors
  to a line bundle $L$ of the family are defined over
  $\mathcal O_S$. We may also assume that the adelic metrics
  outside~$S$ are compatible with the action of the Galois group.
  For any $i$ in $I$, we define the set
  \[\Delta(\torsor_i)=\{\,y\in \torsor_i(\Adeles_\KK)\mid
  \pr((\liftmultih_w(y_w))_{w\in S})\in \Delta\text{ and }
  \forall w\not\in S,\ y_w\in\mathcal \torsor_i(\mathcal O_w)
  \,\},\]
  where $\pr$ is a linear projection on $\ker(\pi)$.
\end{cons}
\begin{lemm}
  For any place $w\not\in S$, the projection map
  $\mathcal \torsor_i(\mathcal O_w)\to V(\KK_w)$
  is surjective and the map $\liftmultih_w$
  is characterized by the following two conditions:
  \begin{conditions}
  \item We have the relation
    $\liftmultih_w(t.y)= -\log_w(t)+\liftmultih_w(y)$
    for any $t\in\TNS(\KK_w)$ and any $y\in\torsor_i(\KK_w)$;
  \item
    The integral points of $\mathcal \torsor_i$ are given by
    \[\mathcal \torsor_i(\mathcal O_w)=\{\,y\in \torsor_i(\KK_w)
    \mid\liftmultih_w(y)=0\,\}.\]
  \end{conditions}
\end{lemm}
\begin{proof}
  Relation (i) follows from the formula for $\Vert t.y\Vert^\LL_w$
  and the description in (ii) from the fact that all maps are compatible with
  the models.
  By the choice of~$S$, for any
  place $w\not\in S$, the projection $\pi_i:\torsor_i(\KK_w)\to V(\KK_w)$
  is surjective. Moreover the functions $\Vert\cdot\Vert^\LL_w$
  are compatible with the action of the Galois group.
  By the diagram \eqref{equ.diag.norm}, it follows that
  $\liftmultih_w(y)$ belongs
  to $\Pic(V_w)\dual$. Since the map $\log_w$ is surjective,
  we may find in any fibre an element $y$ such that $\liftmultih_w(y)=0$.
  By (ii), this element is an integral point.
  Since the map $\pi_i:\mathcal\torsor_i(\mathcal O_w)\to V(\KK_w)$
  is surjective, conditions (i) and (ii) characterize $\liftmultih_w$.
\end{proof}
\begin{theo}\label{theo.lift.count}
  The set $\Delta(\torsor_i)\cap \torsor_i(\KK)$ is a fundamental domain
  for the action of $\TNS(\KK)$ modulo $W(\TNS)$. In other words,
  it satisfies the following conditions:
  \begin{conditions}
  \item
    We have $\torsor_i(\KK)=\cup_{t\in\TNS(\KK)}t.
    \bigl(\Delta(\torsor_i)\cap \torsor_i(\KK)\bigr)$;
  \item
    For any $t\in\TNS(\KK)$, we have
    \[\bigl(\Delta(\torsor_i)\cap \torsor_i(\KK)\bigr)
    \cap t.\bigl(\Delta(\torsor_i)\cap \torsor_i(\KK)\bigr)
    \neq\emptyset\]
    if and only if $t\in W(\TNS)$.
  \item
    For $t\in W(\TNS)$, we have
    \[t.\bigl(\Delta(\torsor_i)\cap \torsor_i(\KK)\bigr)
    =\Delta(\torsor_i)\cap \torsor_i(\KK).\]
  \end{conditions}
\end{theo}
\begin{proof}
  Let $y\in\torsor_i(\KK)$. By the lemma, for any $w\not\in S$,
  $\liftmultih_w(y)\in\Pic(V_w)\dual$. Thus, using the exact
  sequence \eqref{equ.exact.ono}, we get an element $t\in\TNS(\KK)$
  such that $t.y\in\mathcal\torsor_i(\mathcal O_w)$ for $w\not\in S$.
  Using the exact sequence \eqref{equ.exact.ono2} and the definition
  of $\Delta$, there is an element~$t'$ in $\TNS(\mathcal O_S)$
  such that $(t't).y\in\Delta(\torsor_i)$. Assertions (ii) and (iii)
  follow from the definition of $\Delta$.
\end{proof}
\begin{nota}
  For any $i\in I$, we define the map
  \[\liftmultih:\Delta(\torsor_i)\longrightarrow \Pic(V)_\RR\dual\]
  by the relation $\liftmultih(y)=\pi\bigl((\liftmultih_w(y_w))_{w\in S}\bigr)$.
\end{nota}
\begin{theo}%
  \label{theo.lift.constant}
  We assume conditions~\ref{hypos:conditions}.
  Let $W$ be a borelian subset of $V(\Adeles_\KK)$. Let $\mathcal D$ be a
  borelian subset of $\Pic(V)_\RR\dual$. Then
  \begin{multline*}
    \beta(V)\nu(\mathcal D)\oomega_V(W\cap V(\Adeles_\KK)^{\Br})\\
    =
    \frac 1{W(\TNS)}
    \sum_{i\in I}\oomega_{\torsor_i}(\{\,y\in\Delta(\torsor_i)
    \cap\pi_i^{-1}(W)\mid
    \liftmultih(y)\in\mathcal D\,\}).
  \end{multline*}
\end{theo}
\begin{proof}
  We will sketch this proof which
  follows the ideas of Salberger~\cite{salberger:tamagawa}.
  If  $(\xi_1,\dots,\xi_r)$ is a basis of $X^*(\TNS)=\Pic(V_\LL)$,
  then $\bigwedge_{i=1}^r\xi^{-1}_i\Haar{\xi_i}$ is a section of
  $\omega_{\TNS}$, which, up to sign, does not depend on the choice
  of the basis. This defines a canonical Haar measure $\oomega_{\TNS,w}$
  on $\TNS(\KK_w)$ for any place~$w$ of~$\KK$.
  Let $w\in\Val(\KK)\setminus S$. Locally for $w$-adic topology,
  we may choose a section of
  $\pi_i:\mathcal\torsor_i(\mathcal O_w)\to V(\KK_w)$ and the measure
  $\oomega_{\torsor_i,w}$ on $\mathcal\torsor_i(\KK_w)$ is locally isomorphic
  to the measure
  \[L_w(1,X^*(\TNS))|\omega_V(t)|_w\oomega_{\TNS,w}
  \times\lambda_w\oomega_{V,w}.\]
  where $\omega_V$ is seen as a character of $\TNS$.
  Let us also consider the groups
  \[\TNS(\Adeles_\KK)^1=\Bigl\{\,(t_w)_{w\in\Val(\KK)}\in\TNS(\Adeles_\KK)
  \Bigm|
  \forall\xi\in X^*(\TNS),\ \prod_{w\in\Val(\KK)}|\xi(t_w)|_w=1\,\Bigr\}\]
  and
  \[\TNS(\KK_S)^1=\Bigl\{\,(t_w)_{w\in S}\in
  \prod_{w\in S}\TNS(\KK_w)
  \Bigm|
  \forall\xi\in X^*(\TNS),\ \prod_{w\in S}|\xi(t_w)|_w=1\,\Bigr\}.\]
  The lattice $X^*(\TNS)\dual$ normalises the Haar measure
  on $X^*(\TNS)_\RR\dual$ and therefore on the quotient
  $\prod_{w\in S}\TNS(\KK_w)/\TNS(\KK_S)^1$. Using
  the measure $\prod_{w\in S}\oomega_{\TNS,w}$ on the product,
  we get a normalised Haar measure $\omega_{T^1}$ on $\TNS(\KK_S)^1$.
  We consider the fibration
  \[\liftmultih\times\pi_i:\prod_{w\in S}\torsor_i(\KK_w)\longrightarrow
  \Pic(V)_\RR\dual\times\prod_{w\in S}V(\KK_w),\]
  which, over its image, is a principal homogeneous space under
  $\TNS(\KK_S)^1$. By choosing a local adequate section of this fibration,
  we get that the measure $\prod_{w\in S}\oomega_{\torsor_i,w}$
  on $\prod_{w\in S}\torsor_i(\KK_w)$ is the measure
  induced by the product measure 
  $\nu\times\prod_{w\in S}\oomega_{V,w}$ on the
  image and the measure $\omega_{T^1}$ on $\TNS(\KK_S)^1$.
  Taking the product over all places, and multiplying by the normalisation
  terms, we get that
  \begin{multline*}
    \frac 1{\card W(\TNS)}
    \oomega_{\torsor_i}(\{\,y\in\Delta(\torsor_i)\cap\pi_i^{-1}(W)\mid
    \liftmultih(y)\in\mathcal D\,\})\\
    =\tau(\TNS)\nu(\mathcal D)
    \oomega_V(\pi_i(\torsor_i(\Adeles_\KK))\cap W),
  \end{multline*}
  where $\tau(\TNS)$ is the Tamagawa number of $\TNS$, that
  is the normalized volume of the compact
  quotient $\TNS(\Adeles_\KK)^1/\TNS(\KK)$ which is isomorphic
  to the product
  \[\TNS(\KK_S)^1/\TNS(\mathcal O_S)\times\prod_{w\not\in S}\TNS(\mathcal O_w).\]
  By Ono's theorem (\cite[\S3]{ono:tamagawa}), the Tamagawa number of $\TNS$
  is given by
  \[\tau(\TNS)=\frac{\card H^1(\KK,X^*(\TNS))}{\card\sha^1(\KK,\TNS)}\]
  where $\sha^1(\KK,\TNS)=\ker(H^1(\KK,\TNS)\to
  \prod_{w\in\Val(\KK)}H^1(\KK_w,\TNS))$. By definition,
  $\beta(V)=\card H^1(\KK,X^*(\TNS))$.
  To conclude the proof, we use the crucial fact,
  first proven by Salberger, that for any $x\in V(\Adeles_\KK)^{\Br}$,
  the number of $i\in I$ such that $x\in \pi_i(\torsor_i(\Adeles_\KK))$
  is precisely equal to $\card\sha^1(\KK,\TNS)$.
\end{proof}
\begin{listrems}
  \remark
  Using theorems~\ref{theo.lift.count} and \ref{theo.lift.constant},
  we see that the equivalence formula~\eqref{equ:estimateallheights}
  of question~\ref{question:allheights}, reduces to an equivalence
  of the form
  \[\card\{\,y\in\torsor_i(\KK)\cap\Delta(\torsor_i)\mid
  \liftmultih(y)\in\mathcal D_B\,\}\sim\oomega_{\torsor_i}(\{\,
  y\in\Delta(\torsor_i)\mid\liftmultih(y)\in\mathcal D_B\,\})\]
  as $B\to +\infty$.
  \remark
  The conditions $y\in\torsor_i(\mathcal O_w)$ for $w\in\Val(\KK)\setminus S$
  correspond to an integrality condition combined with a $\gcd$ condition.
  For example, if $V$ is a smooth complete intersection of dimension $\geq 3$
  in the projective space $\PP^N_\QQ$, then the unique versal torsor $\torsor$
  is the corresponding cone in $\Affine^{N+1}_\QQ\setminus\{0\}$ and
  the condition $(y_0,\dots,y_N)\in\torsor(\ZZ_p)$ corresponds to
  $(y_0,\dots,y_N)\in\ZZ_p^{N+1}$ and $\gcd(y_0,\dots,y_N)=1$.
  Therefore to reduce to counting integral points in a bounded domain,
  the next step is to use a Moebius inversion formula to remove the $\gcd$
  condition. Such an inversion formula is described
  in \cite[\S2.3]{peyre:cercle}.
  \remark In the preceding description, we were not very careful
  about the choice of the finite set~$S$ of bad places. For practical
  reasons, to use this method, it is in fact more efficient to use
  a small set of bad primes.
  \remark The lifting to the versal torsors has been used in many cases,
  see for example~\cite{breteche:cinq} or \cite{bretechebrowningpeyre:chatelet}.
  For practical reasons, it is often simpler to consider an intermediate
  torsor corresponding to the Picard group $\Pic(V)$
  (see for example the work of K.~Destagnol
  \cite{destagnol:chatelet}). The main difference in
  the new approach described in this section
  is that the domain obtained after lifting does not have
  ``spikes''. In other words, the area of the boundary has a smaller rate
  of growth, which should remove some of the problems encountered when using
  a single height relative to the anticanonical line bundle.
\end{listrems}

\subsection{Varieties of Picard rank one}%
\label{subsection:rankone}
If the rank of $\Pic(V)$ is one, then without loss of generality
formula~\eqref{equ:estimateallheights} is reduced to estimating
a difference of the form
\begin{equation}\label{equ:rankonecase}
  \card (V(\KK)-T)_{H\leq bB}-\card (V(\KK)-T)_{H\leq aB}
\end{equation}
as $B$ goes to infinity, where~$H$ is a height relative to the
anticanonical line bundle and $a,b$ are real numbers with $0<a<b$.
Therefore, in that case, a positive answer to
question~\ref{question:allheights} is true if the principle
of Batyrev and Manin is valid for~$V(\KK)-T$. Similarly the global
equidistribution in the sense of~\ref{globalequidistribution:allheights},
follows from global equidistribution~\ref{globalequidistribution}.
However the knowledge of estimates for the difference~\eqref{equ:rankonecase}
does not gives an estimate for $(V(\KK)-T)_{H\leq B}$, unless
we have a uniform upper bound for the error term.

But several examples of Fano varieties of Picard rank one with acccumulating
subvarieties are known in dimension $\geq 3$ (see the list
given in \cite{browningloughran:toomany}). For example, if we consider
a cubic volume, the projective lines it contains are parametrized by the Fano
surface, which is of general type. Each of these rational lines has
degree 2 and as we shall explain in section~\ref{subsubsection:lines},
these lines give a non negligible contribution to the total number of points
thus contradicting the global equidistribution.
In the case of a smooth complete intersection of two quadrics in $\PP^5$,
the situation is even worse since the projective lines it contains
may be Zariski dense.

This shows that in higher dimension, even in the case of varieties with a
Picard
group of rank one, there might be accumulating subvarieties of codimension
$\geq 2$ which are not detected by heights or line bundles.
Thus one needs to go beyond heights. To help us in that direction
we shall first consider the geometric analogue of this problem.

\section{Geometric analogue}%
\label{section:geometricanalogue}
The geometric analogue of the study of rational points of bounded height is
the study of rational curves of bounded degree. This is a very active subject
in algebraic geometry, and we are going to give a very superficial survey
of some particular aspects of this subject in this section.
In fact, there is a very classical dictionary between number fields,
global fields of positive characteristic and function fields of curves.
To simplify the description, we shall mostly restrict ourselves to morphisms
from $\PP^1_k$ to a variety~$V$ defined over~$k$.

\begin{notas}
  Let~$k$ be a field and let $\mathcal C$ be a smooth geometrically integral
  projective curve over $k$. In this section, we denote
  by $\KK=k(\mathcal C)$ the
  function field of~$\mathcal C$. Let~$V$ be a nice variety over~$k$.
  The image of the generic point gives a bijection between
  the set of rational point $V(\KK)$ and the set of morphisms
  $f:\mathcal C\to V$. From now on, we shall identify these sets.
  Let $f:\mathcal C\to V$ be a point of this space. Then
  the pull-back map is a morphism of groups $f^*:\Pic(V)\to\Pic(\mathcal C)$.
  The composition $\deg\circ f^*$ is an element of $\Pic(V)\dual$,
  which we call the \emph{multidegree of~$f$}\index{Multidegree}
  and denote by $\multideg(f)$.
  \par
  The constructions of Grothendieck \cite{grothendieck:hilbert} prove
  that for any $d\in\Pic(V)\dual$, there exists a variety
  $\HHom^d(\mathcal C,V)$ defined over~$k$, which parametrizes the
  morphisms from $\mathcal C$ to $V$ of multidegree $d$.
\end{notas}
In that geometric setting, we want to describe \emph{asymptotically} the
geometric properties of the variety $\HHom^d(\mathcal C,V)$ as the distance
from $d$ to the boundary of the dual of the effective cone goes to infinity.
The problem is to give a framework for the asymptotic study of a variety.
We shall
use the framework given by the ring of integration which was introduced
by Kontsevich (see also \cite{denefloeser:germs}).

\subsection{The ring of motivic integration}
Of course, the dimension of the variety $\HHom^d(\mathcal C,V)$
goes to infinity as the multidegree $d$ grows. But,
as suggested by the work of J.~Ellenberg, we could consider the stabilisation
of cohomology groups. The ring of motivic integration enables us
to consider the limit of a class associated to the variety.

\begin{cons}
  We denote by $\mM_k$ the Grothendieck ring of varieties over~$k$:
  as a group it is generated by the isomorphism classes of varieties
  over~$k$, where the class of a variety~$V$ is denoted by~$[V]$,
  with the relations
  \[[V]=[F]+[U]\]
  for any closed subvariety~$F$ of~$V$, with $U=V\setminus F$.
  We can then extend the definition of a class to non reduced schemes.
  Then $\mM_k$ is equipped with the unique ring structure
  such that
  \[[V_1]\times[V_2]=[V_1\times_k V_2],\]
  for any varieties~$V_1$ and~$V_2$ over~$k$.
  We define the tate symbol as $\tate=[\Affine^1_k]$
  and consider the localized ring $\mM_{k,\loc}=\mM_k[\tate^{-1}]$.
  We then introduce a decreasing filtration on this ring where, for $i\in\ZZ$,
  \[F^i\mM_{k,\loc}\] is the subgroup of
  $\mM_{k,\loc}=\mM_k[\tate^{-1}]$ generated by symbols of the form
  $[V]\tate^{-n}$ if $\dim(V)-n\leq -i$.
  We have the inclusion
  \[F^i\mM_{k,\loc}.F^j\mM_{k,\loc}\subset F^{i+j}\mM_{k,\loc},\]
  for $i,j\in\ZZ$.
  Thus the inverse limit $\motivic_k=\varprojlim_i\mM_{k,\loc}/F^i\mM_{k,\loc}$
  comes equipped with a structure of topological ring so that the natural map
  $\mM_{k,\loc}\to\motivic_k$ is a morphism of rings.
\end{cons}
\begin{rema}
  The morphism $\mM_k\to\mM_{k,\loc}$ is not injective
  (see \cite{borisov:grothendieck}),
  so we loose information by looking at classes in $\motivic_k$. 
\end{rema}
With this ring we may formulate the analogue of
question~\ref{question:allheights}:
\begin{question}\label{question:allgeometric}
  We assume that the nice variety~$V$ over~$k$ is rationally connected,
  satisfies conditions (i) and (iii) to (v) of
  hypotheses~\ref{hypos:conditions} and that the rational points over $k(T)$
  are Zariski dense.  Does the symbol
  \[\Bigl[\HHom^d(\PP^1_k,V)\Bigr]\tate^{-\langle\omega_V^{-1},d\rangle}\]
  converges in $\motivic_k$ for $d\in\Pic(V)\dual\cap \Ceffdoof V$
  as $\dist(d,\partial\Ceffdof V)$ goes to infinity and can we interpret
  the limit as some adelic volume?
\end{question}
\subsection{A sandbox example: the projective space}
In the case of the projective space, it turns out that the symbol
in fact stabilizes, and thus converges:
\begin{prop}
  If $d\geq 1$, then
  \[\Bigl[\HHom^d(\PP^1_k,\PP^n_k)\Bigr]\tate^{-(n+1)d}=\frac{\tate^{n+1}-1}{\tate-1}
  (1-\tate^{-n}).\]
\end{prop}
\begin{proof}[Sketch of the proof]
  In this proof, we shall describe the sets of $k$-points of our varieties
  and gloss over the description of the varieties themselves.
  So if consider the set $W_d(k)$ of $(P_0,\dots,P_n)\in k[T]^{n+1}$
  such that $\gcd_{0\leq i\leq n}(P_i)=1$ and $\max_{0\leq i\leq n}(\deg(P_i))=d$
  then $W_d$ is a $\GG_m$ torsor over the space $\HHom^d(\PP^1_k,\PP^n_k)$
  which is locally trivial for Zariski topology.
  Hence
  \begin{equation}\label{equ:gmfibration}
    (\tate-1)\Bigl[\HHom^d(\PP^1_k,\PP^n_k)\Bigr]=[W_d].
  \end{equation}
  But if we consider the space of $(n+1)$-tuples of polynomials
  $(P_0,\dots,P_n)$ such that $\max_{0\leq i\leq n}(\deg(P_i))=d$,
  then it is naturally isomorphic to
  $\Affine^{(n+1)(d+1)}-\Affine^{(n+1)d}$ and we may decompose it
  as a disjoint union according to the degree of the $\gcd$ of the
  polynomials. The piece corresponding to the families
  with $\deg(\gcd_{0\leq i\leq n}(P_i))=k$ is isomorphic to
  $[W_{d-k}]\times\Affine^k$ where $\Affine^k$ parametrizes the
  $\gcd$ which is a unitary polynomial of degree $k$.
  We get the formula
  \[\tate^{(n+1)(d+1)}-\tate^{(n+1)d}=\sum_{k=0}^d\tate^k[W_{d-k}].\]
  We may introduce formal series in $\motivic_k[[T]]$ to get
  the formula
  \[\sum_{d\geq 0}(\tate^{n+1}-1)\tate^{(n+1)d}T^d=
  \Bigl(\sum_{k\geq 0}\tate^kT^k\Bigr)\Bigl(\sum_{d\geq 0}[W_d]T^d\Bigr).\]
  From which we deduce
  \[\sum_{d\geq 0}[W_d]T^d=
  (1-\tate T)(\tate^{n+1}-1)\sum_{d\geq 0}\tate^{(n+1)d}T^d.\]
  Therefore, if $d\geq 1$, we get
  \[[W_d]=(\tate^{n+1}-1)(\tate^{(n+1)d}-\tate\tate^{(n+1)(d-1)})
  =(\tate^{n+1}-1)\tate^{(n+1)d}(1-\tate^{-n}).\]
  Combining with formula~\eqref{equ:gmfibration} gives the formula
  of the proposition.
\end{proof}
\begin{listrems}
  \remark
  Let us quickly explain how the constant obtained might be interpreted
  as an adelic volume. First, for the projective space
  the $L$ function associated to the Picard group coincide with
  the usual zeta function.
  This has a motivic analogue decribed by M.~Kapranov in
  \cite{kapranov:motivic}:
  \[Z_{\CC(T)}(U)=\sum_{d\geq 0}[(\PP^1_k)^{(d)}]U^d\]
  where $(\PP^1_k)^{(d)}$ is the symmetric product $(\PP^1_k)^d/\mathfrak S_d$
  and is isomorphic to $\PP^d_k$. The parameter $U$ should be understood as
  $\tate^{-s}$. The residue of the zeta function at $s=1$ corresponds to
  \begin{align*}
    \bigl((1-\tate U)Z_{\CC(T)}(U)\bigr)(\tate^{-1})&=
    \Bigl((1-\tate U)
    \sum_{d\geq 0}\frac{\tate^{d+1}-1}{\tate-1}U^d\Bigr)(\tate^{-1})\\
    &=
    \frac1{\tate-1}\left((1-\tate U)\left(\frac\tate{1-\tate U}-
    \frac1{1-U}\right)\right)(\tate^{-1})\\
    &=\frac1{\tate-1}\left(\frac {\tate-1}{1-U}\right)(\tate^{-1})\\
    &=\frac1{1 -\tate^{-1}}.
  \end{align*}
  By translating the formula~\eqref{equ:adelicmeasure},
  the expected constant should formally have the form
  \[C=\frac{\tate^n}{1-\tate^{-1}}
  \prod_{P\in\PP^1_k}(1-\tate^{-\deg(P)})[\PP^n_{\kappa(P)}]\tate^{-n\deg(P)},\]
  where $\LL^{-1}$ plays the r\^ole of the square root of the descriminant.
  The term appearing in the product may be simplified
  as $1-\tate^{-(n+1)\deg(P)}$.
  However this formal constant involves a product over a
  possibly uncountable set $\PP^1_k$.
  Nevertheless, in this very particular case, we may consider
  the \emph{inverse} of this product. Then, we get
  \[\prod_{P\in\PP^1_k}\sum_{m\geq 0}\tate^{-(n+1)m\deg(P)}\]
  If we admit that it makes sense to develop this product,
  we get
  \[\sum_{m\geq 0}\sum_{P\in(\PP^1_m)^{(m)}}\tate^{-(n+1)m}.\]
  But we may now interpret each interior sum
  as a motivic integral and get
  \begin{align*}
    \sum_{m\geq 0}[\PP^m_k]\tate^{-(n+1)m}
    &=\sum_{m\geq 0}\frac{1-\tate^{m+1}}{1-\tate}\tate^{-(n+1)m}\\
    &=\frac 1{1-\tate}\left(\frac1{1-\tate^{-n-1}}-
    \frac\tate{1-\tate^{-n}}\right)\\
    &=\frac 1{1-\tate}\times
    \frac{1-\tate}{(1-\tate^{-n})(1-\tate^{-n-1})}
  \end{align*}
  Finally we get
  \[C=\frac{\tate^{n+1}-1}{\tate-1}(1-\tate^{-n})\]
  as wanted.
  \remark
  This type of result is compatible with products and we get a result
  for products of projective spaces for free. D. Bourqui has more
  general results for toric varieties~\cite{bourqui:motivique}.
  \remark
  M. Bilu in \cite{bilu:eulerproduct} has defined an Euler
  product giving a precise meaning for the expected constant in
  this setting.
\end{listrems}
\subsection{Equidistribution in the geometric setting}
In the geometric setting equidistribution may be described as follows.
\begin{cons}
  Let $\mathcal S$ be a subscheme of dimension~$0$ of $\mathcal C$,
  then we may consider the moduli space $\HHom(\mathcal S,V)$
  which parametrizes the morphisms from $\mathcal S$ to~$V$.
  For any subvariety~$W$ of $\HHom(\mathcal S,V)$, we may then consider
  the set of morphisms $f:\PP^1_k\to V$ of multidegree $d$ such that
  the restriction $f_{|\mathcal S}$ belongs to $W$. This
  is parametrized by a variety $\HHom^d_W(\mathcal C,V)$ contained
  in $\HHom^d(\mathcal C,V)$.
\end{cons}
\begin{namedtheorem}[Na\"ive geometric equidistribution]
  We shall say that \emph{na\"ive equidistribution}
  holds for $V$ if for any subscheme $\mathcal S$ of dimension~$0$
  in~$\mathcal C$ and any subvariety~$W$ of $\HHom(\mathcal S,V)$,
  the symbol
  \[\left(\Bigl[\HHom^d_W(\mathcal C,V)\Bigr]\Bigl[\HHom(\mathcal S,V)\Bigr]-
    \Bigl[\HHom^d(\mathcal C,V)\Bigr][W]\right)\LL^{-\langle\omega_V^{-1},d\rangle}
  \]
  converges to~$0$ in $\motivic_k$ for
  $d\in\Pic(V)\dual\cap \Ceffdoof V$ as
  $\dist(d,\,\partial\Ceffdof V)$ goes to infinity.
\end{namedtheorem}
\begin{rema}
  This statement gives a precise meaning to the idea of a convergence
  \[\frac{\Bigl[\HHom^d_W(\mathcal C,V)\Bigr]}
         {\Bigl[\HHom^d(\mathcal C,V)\Bigr]}
    \longrightarrow
    \frac{[W]}{\Bigl[\HHom(\mathcal S,V)\Bigr]}.
  \]
  
\end{rema}
\subsection{Crash course about obstruction theory}
Obstruction theory gives a sufficient condition for the moduli
spaces to have the expected dimension. Let us give a very short introduction
to these tools, the interested reader may turn to the book of
O.~Debarre \cite{debarre:higher} for a more serious introduction
to this subject.
\par
let $f:\PP^1_k\to V$ be a morphism of multidegree~$d$
then we may consider the tangent
space at~$f$ and the dimension at~$f$. There is a natural isomorphism
\[T_f\HHom^d(\PP^1_k,V)\longiso H^0(\PP^1_k,f^*(T_V))\]
and
\[\dim_f\bigl(\HHom^d(\PP^1_k,V)\bigr)\geq h^0(\PP^1_k,f^*(T_V))-h^1(\PP^1_k,f^*(T_V)).\]
On the other hand, on $\PP^1_k$, any vector bundle splits into a direct
sum of line bundles. In other words, there exists
an isomorphism
\[f^*(TV)\longiso\bigoplus_{i=1}^n\mathcal O_{\PP^1_k}(a_i)\]
with $a_1\geq a_2\geq\dots\geq a_n$ and $(a_1,\dots,a_n)$ is uniquely
determined. If $a_n\geq 0$, then we get that $h^1(\PP^1_k,f^*(TV))=0$
and
\[\dim_f\bigl(\HHom(\PP^1_k,V)\bigr)=h^0(\PP^1_k,f^*(T_V))=
\sum_{i=1}^nh^0(\mathcal O_{\PP^1_k}(a_i))=\sum_{i=1}^na_i+1
=n+\langle d,\omega_V^{-1}\rangle,\]
which is the expected dimension.
Thus a sufficient condition to get the expected dimension is
$a_n\geq 0$.

But let us now add some conditions related to equidistribution.
Let $\mathcal S$ be a subscheme of $\PP^1_k$ of dimension~$0$.
Then $\mathcal S$ corresponds to a divisor $D=\sum_{P\in I}n_PP$ on $\PP^1_k$
and may described as
$\Spec(\times_{P\in I}\mathcal O_{\PP^1_k,P}/\mathfrak m_P^{n_P})$,
where $\mathfrak m_P$ is the maximal ideal of the local ring
$\mathcal O_{\PP^1_k,P}$. Let $s$ be the
degree of $D$, that is $\sum_{P\in I}n_P[\kappa(P):k]$.
Then $\HHom(\mathcal S,V)$ has dimension $ns$;
therefore if we fix $\varphi:\mathcal S\to V$, the expected dimension
of $\HHom^d_{\{\varphi\}}(\PP^1_k,V)$ ought to be
$n(1-s)+\langle d,\omega_V^{-1}\rangle$.
But obstruction theory in that setting relates the deformation at~$f$
to the vector bundle $f^*(TV)\otimes \mathcal O(-D)$ therefore
the sufficient condition for the dimension of the moduli
space $\HHom^d_{\{\varphi\}}(\PP^1_k,V)$ at~$f$ to be the correct one
is $a_n-s\geq 0$. In other words, a sufficient condition for
the dimensions to be the correct ones is to look at the limit as $a_n$
goes to $+\infty$.

One should note that the counter-examples introduced in
section~\ref{subsection:rankone}, like the intersection of
two quadrics, also show
the necessity to go beyond degrees in the geometric setting. 
\section{Slopes \`a la Bost}%
\label{section:slopes}
Following the geometric analogue, we need a notion which
is the arithmetic traduction of very free curves. This analogue,
introduced in \cite{peyre:liberte}, is
given by Arakelov geometry and is based upon the slopes as they
are considered by J.-B. Bost.
\Subsection{Definition}
\subsubsection{Slopes of an adelic vector bundle over
  \texorpdfstring{$\Spec(\KK)$}{Spec(K)}}
The following definition is a variant of the definition described
in another chapter of this volume.
\begin{defi}
  Let~$E$ be a $\KK$-vector space of finite dimension~$n$
  equipped with
  \begin{itemize}
  \item
    A projective $\mathcal O_\KK$-submodule of constant rank~$n$;
  \item
    For any complex place $w\in\Val(\KK)$, a map
    \[\Vert\cdot\Vert_w:E_w=E\otimes_\KK\KK_w\longrightarrow\RRp\]
    such that there exists a positive definite hermitian form $\phi$ on $E_w$
    so that $\Vert y\Vert_w=\phi(y,y)$;
  \item
    For any real place $w\in\Val(\KK)$ a euclidean norm
    \[\Vert\cdot\Vert_w:E_w\longrightarrow\RRp.\]
  \end{itemize}
  Let~$F$ be a vector subspace of $E$. We equip it with
  $\Lambda_F=\Lambda\cap F$ and the restrictions of the norms.
  The \emph{Newton polygon}\index{Newton polygon}, which we denote by
  $\mathcal P(E)$\glossary{$\mathcal P(E)$: Newton polygon} is
  defined as the convex hull of the set of pairs $(\dim(F),\dega(F))$
  where~$F$ describes the set of vector subspaces of~$E$.
\end{defi}
\begin{rema}
  Let us assume that $\KK=\QQ$.
  If we consider the subspaces~$F$ of dimension~$1$, then $\dega(F)$
  is given as $-\log(\Vert y_0\Vert_\infty)$ where $y_0$
  is a generator of $\Lambda\cap F$. Thus we get the points
  $(1,-\log(\Vert y\Vert_\infty))$ where~$y$ goes over the primitive
  elements of the lattice $\Lambda$. In particular, there is an upper
  bound for the possible values of the second coordinate.
  More generally $\mathcal P(E)$ is bounded from above. In the
  drawing~\ref{figure.convex}, we represented how the points
  $(\dim(F),\dega(F))$ and the upper part of the convex hull may look like.
\end{rema}
\begin{figure}[ht]
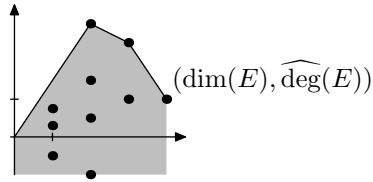

  \centering
  \ifcolorcover
  \includegraphics{peyre_convex-1.mps}
  \else
  \includegraphics{peyre_convex-2.mps}
  \fi
  \caption{Convex hull}
  \label{figure.convex}
\end{figure}
\begin{cons}
  Since the set $\mathcal P(E)$ is bounded from above, we may define
  the function $m_E:[0,n]\to\RR$ by
  \[m_E(x)=\max\{\,y\in\RR\mid (x,y)\in\mathcal P(E)\,\}.\]
  This function is concave and affine in each interval $[i-1,i]$
  for $i\in\{1,\dots,\dim(E)\}$. The slopes of~$E$ are then given
  as
  \[\mu_i(E)=m_E(i)-m_E(i-1)\]
  for $i\in\{1,\dots,\dim(E)\}$.
\end{cons}
\begin{listrems}\label{listrems:slopes}
  \remark
  By construction, we have the inequalities
  \[\mu_1(E)\geq\mu_2(E)\geq\dots\geq\mu_{\dim(E)}(E).\]
  These inequalities might not be strict. Moreover
  \[\dega(E)=\sum_{i=1}^{\dim(E)}\mu_i(E).\]
  Therefore the \emph{slope}\index{Slope} of~$E$,
  which is defined as $\mu(E)=\frac{\dega(E)}{\dim(E)}$ is
  the mean of the slopes:
  \[\mu(E)=\frac1{\dim(E)}\sum_{i=1}^{\dim(E)}\mu_i(E).\]
  \remark
  The value of $m_E(i)$ may differ from
  $\max_{\dim(F)=i}(\dega(F))$. However, following E.~Gaudron
  \cite[definition 5.18]{gaudron:pentes}, we may define the successive
  minima of the arithmetic lattice~$E$ as follows: for
  $i\in\{1,\dots,\dim(E)\}$, the $i$-th minima $\lambda_i(E)$
  is the infimum of the numbers $\theta\in\RRpp$ such that
  there exists a family of strictly positive real numbers
  $(\theta_w)_{w\in\Val(\KK)}$ and a free family $(x_1,\dots,x_i)$ in $E$
  such that
  \begin{conditions}
  \item
    The set $\{\,w\in\Val(\KK)\mid\theta_w\neq 1\,\}$ is finite;
  \item
    The product $\prod_{w\in\Val(\KK)}\theta_w$ is equal to $\theta$;
  \item
    We have the inequalities
    \[\Vert x_j\Vert_w\leq \theta_w\]
    for $j\in \{1,\dots,i\}$ and $w\in\Val(\KK)$.
  \end{conditions}
  Then Minkowski's theorem gives an explicit constant $C_\KK$ such that
  \[0\leq \log(\lambda_i(E))+\mu_i(E)\leq C_\KK\]
  for $i\in\{\,1,\dots,\dim(E)\,\}$.
  \remark
  In this chapter, the slopes are not invariant under field
  extensions since we did not normalise them by $\frac 1{[\KK:\QQ]}$.
  This conforms to the usual convention for heights in Manin's program,
  which has been chosen to get a formulation of
  the expected estimate which does not depend on the degree of the field.
\end{listrems}

\subsubsection{Slopes on varieties, freeness}
We now apply the constructions of last paragraph to vector bundles
on varieties.
\begin{defi}
  Let~$E$ be a vector bundle on the nice variety~$V$ of dimension~$n$.
  We assume that~$E$ is equipped with an adelic norm
  $(\Vert\cdot\Vert_w)_{w\in\Val(\KK)}$ then for any rational
  point $P\in V(\KK)$, the fibre $E_P$ is an adelic vector bundle
  over $\Spec(\KK)$ and we may define
  \[\mu_i^E(P)=\mu_i(E_P).\]
  \par
  In particular, if~$V$ is equipped with an adelic metric, we
  may define the \emph{slopes}%
  \index{Slope>of a rational point}
  of a rational point $P\in V(\KK)$ as
  \[\mu_i(P)=\mu_i(T_PV)\]
  for $i\in\{1,\dots,n\}$.
\end{defi}
\begin{listrems}
  \remark
  From remark \ref{listrems:slopes} (i), we deduce that
  for any rational point $P\in V(\KK)$, we have
  \[\mu_n(P)\leq \mu_{n-1}(P)\leq \dots\leq \mu_1(P)\]
  and $\dega(T_PV)=\sum_{k=1}^n\mu_i(P)$.
  But we may interpret this degree $\dega(T_PV)=\dega((\omega_V^{-1})_P)$
  as the logarithmic height of~$P$, that is $h(P)=\log(H(P))$, where
  the height~$H$ is defined by the induced metric on the anticanonical
  line bundle.
  \remark
  From the previous remark we deduce the inequalities
  \[\mu_n(P)\leq \frac{h(P)}n\leq \mu_1(P)\]
  for any rational point $P\in V(\KK)$.
\end{listrems}
\begin{defi}
  The \emph{freeness}\index{Freeness} of a rational $P\in V(\KK)$
  is defined by
  \[l(P)=
  \begin{cases}
    n\frac{\mu_n(P)}{h(P)}\text{ if $\mu_n(P)>0$},\\
    0\text{ otherwise}.
  \end{cases}
  \]
\end{defi}
\begin{listrems}
  \remark
  By definition the freeness of a point $l(P)$ belongs to
  the interval $[0,1]$.
  \remark
  We have the equality $l(P)=0$ if and only if the minimal slope
  $\mu_n(P)\leq 0$.
  \remark
  The equality $l(P)=1$ occurs if and only if the lattice $T_PV$
  is semi-stable, that is $\mu_1(P)=\dots=\mu_n(P)$. In other words
  this means that $\mu(F)\leq\mu(T_PV)$ for any subspace $F$ of $T_PV$.
  This is, for example, the case if the lattice is the usual lattice
  $\ZZ^n$ in $\RR^n$ equipped with its standard euclidean structure.
  Up to scaling, this occurs for a point $(P,\dots,P)$ on the diagonal
  of $(\PP^1_\KK)^n$. Another example of a semi-stable lattice
  in dimension~$2$ is the classical hexagonal lattice $\ZZ[j]$
  generated by a primitive third root of~$1$, as shown
  in figure~\ref{figure.hexagonal}.
  \begin{figure}[ht]
    \centering
    \includegraphics{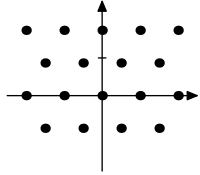}
    \caption{Hexagonal lattice}
    \label{figure.hexagonal}
  \end{figure}
  More generally for two dimensional
  lattices we may consider that $\Lambda$ is
  isomorphic to the lattice $a(\ZZ+\ZZ\tau)\subset\CC$, where
  $\Re(\tau)\in[-1/2,1/2]$, $|\tau|\geq 1$ and $\Im(\tau)>0$.
  Then a lattice is semistable if and only if $\Im(\tau)\leq 1$, which is drawn
  in \ifcolorcover red\else grey\fi\ on figure~\ref{figure.fundamental}.
  \begin{figure}[ht]
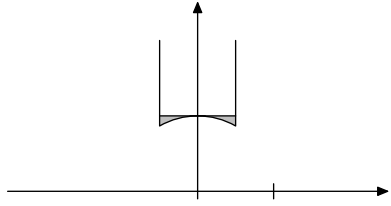

    \centering
    \ifcolorcover
    \includegraphics{peyre_lattices-2.mps}
    \else
    \includegraphics{peyre_lattices-3.mps}
    \fi
    \caption{Semi-stable lattices}
    \label{figure.fundamental}
  \end{figure}
  \remark
  For any rational point on a curve, we have $l(P)=1$.
  \remark\label{rema:slopes.surfaces}
  For a surface~$S$ over~$\QQ$, an adelic metric define
  two invariants, namely
  the height~$H$ and a map $S(\QQ)\to\Poincare/\PSL_2(\ZZ)$,
  where~$\Poincare$ denotes the Poincar\'e half-plane
  $\{\,z\in\CC\mid\Im(z)>0\,\}$ which sends a point $P$ to
  the class of $\tau_P$ such that the lattice in $T_PS$
  is isomorphic to $a_P(\ZZ+\ZZ\tau_P)$. Then the freeness of $P$ is
  given by
  \[l(P)=
  \begin{cases}
    1\text{ if $\Im(\tau_P)\leq 1$},\\
    1-\frac{\log(\Im(\tau_P))}{h(P)}\text{ if $1<\Im(\tau_P)<h(P)$,}\\
    0\text{ otherwise.}
  \end{cases}
  \]
  \remark
  By definition, the freeness $l(P)$ is invariant under field extensions.
  Thus a condition of the form $l(P)>\varepsilon$ does not depend
  on the field of definition and makes sense for algebraic points
  in $V(\overline\KK)$. On the other hand the defining condition
  for a thin subset, namely $P\in\varphi(X(\KK))$
  for a morphism $\varphi$ as in definition~\ref{defi:thin}
  does not make sense for algebraic points.
\end{listrems}

\subsection{Properties}
Let us first describe how the freeness depends on the choice of the metric.
\begin{prop}
  Let $\varphi:E\to F$ be a morphism of vector bundles and let
  ${(\Vert\cdot\Vert_w)_{w\in\Val(\KK)}}$
  (\resp ${(\Vert\cdot\Vert'_w)_{w\in\Val(\KK)}}$) be an adelic norm on~$E$
  (\resp $F$) then there exists a family $(\lambda_w)_{w\in\Val(\KK)}$
  such that
  \begin{conditions}
  \item
    For any $w\in\Val(\KK)$, any $P\in V(\KK_w)$, and any $y\in E_P$,
    we have
    \[\Vert\varphi(y)\Vert'_w\leq\lambda_w\Vert y\Vert_w;\]
  \item
    The set $\{\,w\in\Val(\KK)\mid\lambda_w\neq 1\,\}$ is finite.
  \end{conditions}
\end{prop}
\begin{proof}[Sketch of the proof]
  Let $\PP(E)$ be the projective bundle of the lines in~$E$
  and $E^\times$ be the complement of the zero section
  in $E$. Then for any place~$w$ of~$\KK$,
  we may define a map $f_w:E^\times(\KK_w)\to\RRp$ by
  $f_w(y)=\frac{\Vert\varphi(y)\Vert'_w}{\Vert y\Vert_w}$. This
  map is constant on the lines and induces a continuous map
  $\PP(E)(\KK_w)\to\RRp$. Since the space $\PP(E)(\KK_w)$ is
  compact, this function is bounded from above by a constant $\lambda_w$.
  Moreover for almost all $w\in\Val(\KK)$ the norms on $E$ and $F$
  are defined by model and the morphism $\varphi$ is defined
  over $\mathcal O_w$. For such a place $w$, for any $P\in V(\KK_w)$,
  we get that
  \[\varphi(\{\,y\in E_P\mid\Vert y\Vert_w\leq 1\,\})
  \subset\{\,y\in F_P\mid\Vert y\Vert'_w\leq 1\,\},\]
  therefore we may take $\lambda_w\leq 1$.
\end{proof}
\begin{rema}
  From this lemma, it follows that, if ${(\Vert\cdot\Vert_w)_{w\in\Val(\KK)}}$
  and ${(\Vert\cdot\Vert'_w)_{w\in\Val(\KK)}}$ are norms on a vector bundle,
  then the quotient $\frac{\Vert\cdot\Vert'_w}{\Vert\cdot\Vert'_w}$
  is bounded from above and from below by a strictly positive constant.
  Moreover, by definition the norms are equal for almost all places.
  This implies the existence of a constant $C$ such that,
  for any rational point $P\in V(\KK)$ and
  any subspace~$F$ of $T_PV$,
  \[|\dega(F)-\dega\nolimits'(F)|\leq C.\]
  where $\dega'$ is the degree corresponding to the second norm.
\end{rema}
\begin{coro}\label{coro:changenorm}
  Let $\mu_i$ and $\mu'_i$ be the slopes defined by two different metrics
  on~$V$ and let~$l$ and~$l'$ be the corresponding freeness, then
  \begin{conditions}
  \item
    The difference $|\mu_i-\mu'_i|$ is bounded on $V(\KK)$;
  \item
    There exists $C\in\RRpp$ such that
    \[|l(P)-l'(P)|<\frac C{h(P)}\]
    for any $P\in V(\KK)$ such that $h(P)>0$.
  \end{conditions}
\end{coro}
We now wish to describe a strong link between the geometric and
arithmetic settings. Let us first define the freeness in the
geometric setting.
\begin{defi}
  Let $\varphi:\PP^1_\KK\to V$ be a morphism of varieties.
  The pull-back of the tangent bundle
  $\varphi^*(TV)$ is isomorphic to a direct sum
  $\bigoplus_{i=1}^n\mathcal O_{\PP^1_\KK}(a_i)$ with
  $a_1\geq a_2\geq\dots\geq a_n$. The slopes of~$\varphi$ are the
  integers $\mu_i(\varphi)=a_i$. We may consider
  $\deg_{\omega_V^{-1}}(\varphi)=\sum_{i=1}^n\mu_i(\varphi)$
  and the \emph{freeness of~$\varphi$} is defined by
  \[l(\varphi)=
  \begin{cases}
    \frac{na_n}{\deg_{\omega_V^{-1}}(\varphi)}\text{ if $a_n>0$,}\\
    0\text{ otherwise.}
  \end{cases}
  \]
\end{defi}
\begin{rema}
  By construction $l(\varphi)\in[0,1]\cap\QQ$ and $l(\varphi)>0$
  if and only if $\varphi$ is very free.
\end{rema}
\begin{prop}\label{prop:rational}
  Let $\varphi:\PP^1_\KK\to V$ be a non constant
  morphism of varieties and assume
  that~$V$ is equipped with an adelic metric. Then
  \[l(\varphi(P))\longrightarrow l(\varphi)\]
  as $h_{\PP^1_\KK}(P)\to +\infty$.
\end{prop}
\begin{proof}
  Let us fix an isomorphism from $\varphi^*(TV)$ to
  a direct sum $\bigoplus_{i=1}^n\mathcal O_{\PP^1_\KK}(a_i)$
  with $a_1\geq a_2\geq\dots\geq a_n$. On $\varphi^*(TV)$ we
  consider the pull-back of the adelic metric on~$V$ and we equip
  the sum $\bigoplus_{i=1}^n\mathcal O_{\PP^1_\KK}(a_i)$ with the direct
  sums of the norms induced by a norm on $\mathcal O_{\PP^1_\KK}(1)$.
  Using the corollary~\ref{coro:changenorm}, we get that
  the differences $|\mu_i(\varphi(P))-a_ih_{\PP^1_\KK}(P)|$ is
  bounded, as well as $|h(\varphi(P))-\sum_{i=1}^na_ih_{\PP^1_\KK}(P)|$.
  If $a_n\geq 0$, then the sum $\sum_{i=1}^na_i$ is strictly positive
  since the morphism is not constant
  and we get
  \[\left|l(\varphi(P))-\frac{a_nn}{\sum_{i=1}^na_i}\right|<
  \frac C{h_{\PP^1_\KK}(P)}.\]
  If $a_n<0$, then we get that $l(\varphi(P))= 0$ except for a finite number
  of $P\in\PP^1_\KK$.
\end{proof}

\Subsection{Explicit computations}
\subsubsection{In the projective space}
Let us compute the freeness for points of the projective space:
\begin{prop}
  Let $P\in\PP^n(\KK)$, then
  \[l(P)=\frac n{n+1}+\min_F\left(\frac{-n\dega(F)}{\codim_E(F)h(P)}\right)\]
  where~$F$ goes over the subspaces $F\subsetneq E$ such that $P\in\PP(F)$.
\end{prop}
\begin{proof}[Sketch of the proof]
  Let~$D\subset E$ be the line in~$E$ corresponding to the projective
  point~$P$. There is a canonical isomorphism from the tangent space
  $T_P\PP^n_\KK$ to the quotient $D\dual\otimes E/D\dual\otimes D$
  where~$D\dual$ is the dual of~$D$. This gives a bijection from
  the set of subspaces~$F$ of~$E$ such that $D\subset F\subsetneq E$
  to the strict subspaces of $T_P\PP^n_\KK$ which maps
  the subspace~$F$ to the quotient $D\dual\otimes F/D\dual\otimes D$.
  Since $D\dual\otimes D$ is canonically isomorphic to~$\KK$,
  the arithmetic degree of the subspace of $T_P\PP^n_\KK$ is given by
  \[\dega(D\dual\otimes F/D\dual\otimes D)=
  \dega(D\dual\otimes F)-\dega(\KK)=\dega(F)-\dim(F)\dega(D).\]
  On the other hand, by the description of the tangent space,
  \[h(P)=-(n+1)\dega(D).\]
  We get that the smallest slope is given by
  \[\mu_n(P)=-\dega(D)+\min_F\left(\frac{-\dega F}{\codim_E(F)}\right)\]
  and the freeness by
  \[l(P)=\frac n{n+1}+\min_F
  \left(\frac{-n\dega(F)}{\codim_E(F)h(P)}\right).\qed\]
  \noqed
\end{proof}
\begin{coro}
  For any point $P\in\PP^n(\KK)$, we have
  \[l(P)\geq \frac n{n+1}.\]
\end{coro}
\begin{listrems}
  \remark
  If we take a fixed projective subspace $F$ in $E$,
  then $l(P)$ converges to $\frac n{n+1}$ as $h(P)$ goes to
  $+\infty$ with $P\in F$.
  \remark
  One can show that for any $\eta>0$, there exists a constant $C>0$
  such that, for $B>1$,
  \[\card\{\,P\in\PP^n(\KK)\mid
  H(P)\leq B\text{ and }l(P)<1-\eta\,\}<CB^{1-\eta}.\]
  Since we have an equivalence
  \[\card\{\,P\in\PP^n(\KK)\mid H(P)\leq B\,\}\sim
  C(\PP^n_\KK)B\]
  as~$B$ goes to infinity,
  this means that the number of points~$P$ with a freeness $l(P)<1-\eta$
  is in fact asymptotically negligible. 
\end{listrems}
\subsubsection{Products of lines}
Despite the previous example, the freeness of points can be very small even on
a homogeneous variety. Let us prove that for $(\PP^1_\KK)^n$.
\begin{prop}\label{prop:freenessproduct}
  We equip $(\PP^1_\KK)$ with the product of the adelic metrics and
  denote by $h$ the logarithmic height corresponding to $\PP^1_\KK$.
  then for~$\boldsymbol P=(P_1,\dots,P_n)\in\PP^1(\KK)^n$
  \[l(\boldsymbol P)=\frac{n\,\min_{1\leq i\leq n}(h(P_i))}{\sum_{i=1}^nh(P_i)}.\]
\end{prop}
\begin{proof}
  The tangent space $T_{\boldsymbol P}(\PP^1_\KK)^n$
  is canonically isomorphic to $\bigoplus T_{P_i}\PP^1_\KK$.
  Let us choose a permutation $\sigma\in\mathfrak S_n$ such that
  \[h(P_{\sigma(1)})\geq h(P_{\sigma(2)})\geq\dots\geq h(P_{\sigma(n)}).\]
  Then we get that $\mu_i(\boldsymbol P)=h(P_{\sigma(i)})$, since
  the the subspace of dimension~$i$ with the biggest arithmetic degree is
  given by $\bigoplus_{j=1}^iT_{P_{\sigma(j)}}\PP^1_\KK$.
\end{proof}
\begin{coro}
  For any $\varepsilon>0$,there exist a constant $C_\varepsilon$ such that
  \[\frac{\card\{\,P\in\PP^1(\KK)^n\mid H(P)\leq B\text{ and }l(P)
    >\varepsilon\,\}}
         {\card\{\,P\in\PP^1(\KK)^n\mid H(P)\leq B\,\}}
         \longrightarrow C_\varepsilon\]
  as $B\to+\infty$. Moreover $1-C_\varepsilon=O(\varepsilon)$.
\end{coro}
\begin{proof}[Sketch of the proof]
  Let us consider the map $\boldsymbol h:\PP^1(\KK)^n\to\RRp^n$
  given by $(P_i)_{1\leq i\leq n}\mapsto(h(P_i))_{1\leq i\leq n}$
  and, for $\boldsymbol t=(t_i)_{1\leq i\leq n}$, write
  $|\boldsymbol t|=\sum_{i=1}^nt_i$.
  The height of point $\boldsymbol P$ in $\PP^1(\KK)^n$
  is given by $h(\boldsymbol P)=|\boldsymbol h(\boldsymbol P)|$.
  By proposition~\ref{prop:freenessproduct}, we only have to estimate
  the cardinal of the set
  \[\left\{\,(P_i)_{1\leq i\leq n}\in\PP^1(\KK)^n\left|
  \sum_{i=1}^nh(P_i)\leq
  \min\left(\log(B),\frac n\varepsilon\min_{1\leq i\leq n}(h(P_i))
  \right)\,\right.\right\}.\]
  Let us introduce the compact simplex $\Delta_\varepsilon(B)$ in $\RRp^n$
  defined by
  \[|\boldsymbol t|\leq
  \min\left(\log(B),\frac n\varepsilon\min_{1\leq i\leq n}(t_i)\right).\]
  Then we may write the above set as
  \[\{\,\boldsymbol P\in\PP^1(\KK)^n\mid\boldsymbol h(\boldsymbol P)\in
  \Delta_\varepsilon(B)\,\}.\]
  Using the estimate of E. Landau~\cite{landau:elementare}
  \[\card\{\,P\in\PP^1(\KK)\mid H(P)\leq B\,\}=
  C(\PP^1_\KK)B+O(B^{1/2}\log(B)),\]
  we get that, for real numbers $\eta,\delta$ with $0<\eta<1$
  and $0<\delta<1/2$ and any $\boldsymbol t=(t_1,\dots,t_n)\in\RRp^n$, we have
  \begin{equation}
    \begin{aligned}
      &\card\left\{\,\boldsymbol P\in\PP^1(\KK)^n\left|
      \boldsymbol h(\boldsymbol P)\in\prod_{i=1}^n[t_i,t_i+\eta]
      \right.\,\right\}\\
      &=C(\PP^1_\KK)^ne^{|\boldsymbol t|}(e^\eta-1)^n
      +O(e^{|\boldsymbol t|-\delta\min_{1\leq i\leq n}(t_i)})\\
      &=C(\PP^1_\KK)^ne^{|\boldsymbol t|}\eta^n + O(e^{|\boldsymbol t|}\eta^{n+1})
      +O(e^{|\boldsymbol t|-\delta\min_{1\leq i\leq n}(t_i)}).
    \end{aligned}
  \end{equation}
  Covering $\Delta_\varepsilon(B)$ with cubes with edges
  of length $\eta$, the number
  of such cubes meeting the boundary of the simplex is bounded by
  $O((\log(B)/\eta)^{n-1})$. Therefore comparing sum and integral, we get
  the following estimate for the cardinal of our set:
  \[C(\PP^1_\KK)^n\int_{\Delta_\varepsilon(B)}e^{|\boldsymbol t|}\Haar{\boldsymbol t}
  +O(B(\log(B))^n\eta)+O\left(\left(\frac{\log(B)}
  \eta\right)^nB^{1-\delta\varepsilon/n}\right).\]
  We may take $\eta=B^{-\delta\epsilon/(2n^2)}$ to have a sufficiently small error
  term. The computation of the integral gives $BP_\varepsilon(\log(B))$
  where $P_\varepsilon$ is a polynomial of degree $n-1$
  and leading coefficient $\frac 1{(n-1)!}+ O(\varepsilon)$.
  To conclude, we note that $C((\PP^1_\KK)^n)=\frac 1{(n-1)!}C(\PP^1_\KK)^n$.
\end{proof}
\begin{listrems}
  \remark
  The proof shows that the number of points with freeness $<\varepsilon$
  is not negligible in this case!
  \remark
  If we consider as in section~\ref{section:allheights} the points~$P$
  in $\PP^1(\KK)^n_{h\in\mathcal D_B}$ where
  $\mathcal D_B=\mathcal D_1+\log(B)u$, with $u=(u_i)_{1\leq i\leq n}$, then
  \[l(P)\longrightarrow\frac{n\min_{1\leq i\leq n}(u_i)}{\sum_{i=1}^nu_i}\]
  as~$B$ goes to infinity. Thus, in this case, the set
  \[\{\,P\in V(\KK)\mid h(P)\in\mathcal D_B,\ l(P)<\varepsilon\,\}\]
  is empty for $B$ big enough.
\end{listrems}
\subsection{Accumulating subsets and freeness}
We are now going to show that the freeness gives valuable information
about points related to accumulating phenomena.
\subsubsection{Rational curves of low degree}\label{subsubsection:lines}
Conjecturally the accumulating subsets on projective
surfaces are rational curves
of low degree. More precisely, the number of points on a rational
curve~$L$ in a nice variety~$V$ for a height given by an adelic metric
is equivalent to $C(L)B^{2/\langle L,\omega_V^{-1}\rangle}$. Therefore
such a curve would be accumulating if $\langle L,\omega_V^{-1}\rangle < 2$
and could be weakly accumulating if $\langle L,\omega_V^{-1}\rangle =2$
and the rank of the Picard group of the variety is~$1$.
On a surface, by the adjunction
formula,
\[-2=\deg(\omega_{L})=\langle L,L\rangle+\langle L,\omega_S\rangle.\]
If the rank of the Picard group $\Pic(V)$ is one,
any effective divisor is ample since $S$ is projective,
in that case $\langle L,L\rangle>0$, hence
$\langle L,\omega_S^{-1}\rangle>2$ which excludes the last case
for a surface.
The remaining cases
are covered by the following proposition.
\begin{prop}
  Let~$V$ be a nice variety on the number field~$\KK$,
  and let~$L$ be a rational curve in~$V$ such that
  $\langle L,\omega_V^{-1}\rangle<2$. Then the set
  \[\{\,P\in L(\KK)\mid l(P)>0\,\}\]
  is finite.
\end{prop}
\begin{proof}
  Choose a morphism $\varphi:\PP^1_\KK\to L$ which is birational
  and an isomorphism
  $\varphi^*(TS)\iso\bigoplus_{i=1}^n\mathcal O_{\PP^1_\KK}(a_i)$
  with $a_1\geq a_2\geq\dots\geq a_n$. Then $\mu_i(\varphi)=a_i$
  and $\sum_{i=1}^n\mu_i(\varphi)=\langle L,\omega_V^{-1}\rangle<2$.
  We have a natural morphism $T\PP^1_\KK\to\varphi^*(TV)$
  which implies that $a_1\geq 2$ Therefore $a_2<0$ and we may apply
  proposition~\ref{prop:rational}.
\end{proof}
\begin{listrems}
  \remark
  If we consider only the rational points which satisfy the condition
  $l(P)>\varepsilon(B)$ for some decreasing function
  $\varepsilon$ with values in $\RRpp$, then we exclude all points
  of~$L$ outside a finite set.
  \remark
  In dimension $\geq 3$, if $\langle L,\omega_V^{-1}\rangle = 2$,
  then we get that the freeness~$l(P)$ goes to~$0$ on~$L$.
  This applies to the projective lines in cubic volumes
  or complete intersections
  of two quadrics in $\PP^6$.
\end{listrems}

\subsubsection{Fibrations}
We remind the reader that, in the counter-example of Batyrev and Tschinkel,
the accumulating subset is the reunion of fibers of a fibration. We
are now going to explain that the freeness also detects such abnormality.
\begin{prop}
  Let $\varphi: X\to Y$ be a dominant morphism of nice varieties.
  Then there exists a constant~$C$ such that for any $P\in X(\KK)$ such
  that the linear map $T_P\varphi$ is onto,
  \[\mu_{\dim(X)}(P)\leq \mu_{\dim(Y)}(\varphi(P))+C.\]
  If, moreover, the logarithmic height of~$P$ is strictly positive,
  we get the inequality:
  \[l(P)\leq \frac{mh(\varphi(P))}{nh(P)}l(\varphi(P))+
  \frac{mC}{h(P)}\]
  with $m=\dim(X)$ and $n=\dim(Y)$.
\end{prop}
\begin{proof}
  The linear map $T_P\varphi$ induces a dual map
  $T_P\varphi\dual:T_{\varphi(P)}Y\dual\to T_PX\dual$
  which is injective. We get an inequality
  \[\mu_1(T_{\varphi(P)}Y\dual)\leq \mu_1(T_PX\dual)
  +\max_{1\leq k\leq\dim(Y)}\left(\frac{\log\left(\vertiii{
      \bigwedge^kT_P\varphi\dual}\right)}{k}\right)\leq \mu_1(T_PX\dual)+C.\]
  We conclude with the duality formula for slopes.
\end{proof}
\begin{coro}
  Let $Q\in Y(\KK)$ be a non critical value of $\varphi$, then
  $l(P)$ converges to $0$ as $h(P)$ goes to $+\infty$ with~$P$
  in the fibre $X_Q(\KK)$.
\end{coro}
\begin{rema}
  In particular, this detects bad points in the counter-example
  of Batyrev and Tschinkel. Of course this result applies to $(\PP^1_\KK)^2$
  as well. In fact it is the very property which makes freeness efficient
  to detect bad points in the counter-example of Batyrev and Tschinkel
  which implies that the proportion of rational points in
  $(\PP^1_\KK)^2$ with small freeness is not negligible.
  Section~\ref{section:local}
  will show how the freeness reveals subvarieties which are locally
  accumulating even if they are not globally accumulating.
\end{rema}

\subsection{Combining freeness and heights}
To conclude this part, let us suggest a formula which takes into
account both the freeness and all the heights.
\begin{defi}
  Let $\mathcal D_1$ be a compact
  polyhedron in $\Pic(V)_\RR\dual$ and let $u\in\Ceffdoof{V}$. 
  For any $B>1$ we define $\mathcal D_B=\mathcal D_1+\log(B)u$.
  Let $\varepsilon\in\RRpp$ be small enough, relatively to the
  distance from $u$ to the boundary of $\Ceffdof{V}$.
  Then we define
  \[V(\KK)^{l>\varepsilon}_{\multih\in\mathcal D_B}
  =\{\,P\in V(\KK)\mid
  \multih(P)\in \mathcal D_B,l(P)>\varepsilon\,\}.\]
\end{defi}
Instead of using a constant $\varepsilon$, we could also consider a
slowly decreasing function in $B$ as in~\cite{peyre:liberte}.
With these notations, we can ask our final questions:
\begin{question}\label{question:allheights.freeness}
  We assume that our nice variety~$V$ satisfies the conditions
  of the hypothesis~\ref{hypos:conditions}. Do we have an equivalence
  \begin{equation}
    \label{equ:estimateallheights.freeness}
    \card V(\KK)^{l>\varepsilon}_{h\in\mathcal D_B}\sim \beta(V)\nu(\mathcal D_1)
      \oomega_V(V(\Adeles_\KK)^{\Br})B^{\langle\omega_V^{-1},u\rangle}
  \end{equation}
  as $B$ goes to infinity?
\end{question}
\begin{namedtheorem}[Equidistribution]%
  \label{globalequidistribution:allheights.freeness}
  We shall say that \emph{free points are equidistributed for $\multih$}%
  \index{Global equidistribution>for systems of heights}%
  \index{Equidistribution>Global=(Global ---)}
  if
  the measure $\ddelta_{V(\KK)^{l>\varepsilon}_{\multih\in\mathcal D_B}}$ converges
  weakly to $\mmu_V^{\Br}$ as $B$ goes to infinity.
\end{namedtheorem}
\section{Local accumulation}%
\label{section:local}
The rational points on $\PP^2_\KK$ and $(\PP^1_\KK)^2$
are equidistributed in the sense
of na\"ive equidistribution~\ref{defi:naiveequidistribution}.
But if one looks at figures~\ref{figure.plane} and~\ref{figure.product},
we see lines, which are all projective lines
for the projective plane and the fibres of the two projections
for the product of two projective lines. To interpret these lines,
we need to go beyond the global distribution.
\subsection{Local distribution}
Let us assume that $\KK=\QQ$ to simplify the discussion.
Instead of looking at the proportion of points in a fixed
open subset~$U$ in the adelic space, we may look at the rational
points of bounded height in a open subset $U_B$ depending on $B$
and ask the very broad question
\begin{question}
  For which families $(U_B)_{B>1}$ of open subsets in $V(\Adeles_\QQ)$
  can we hope to have
  \[\frac{\card U_B\cap V(\QQ)_{\multih\in\mathcal D_B}}
         {\card V(\QQ)_{\multih\in\mathcal D_B}}\sim
  \mmu^{\Br}_V(U_B)\]
  as~$B$ goes to infinity?
\end{question}
A particularly interesting case is the distribution around a rational point.
Fix $P_0\in V(\QQ)$
and choose a local diffeomorphism $\rho:W\to W'$, where~$W$ is an open
subset in $V(\RR)$ and $W'$ is an
open subset of $T_{P_0}V_\RR$, which maps $P_0$
to $0$ and such that the differential at $P_0$ is the identity map.
Then we may try to zoom in on the point $P_0$ with some power of~$B$.
More precisely, let us consider the ball
\[\mathcal B(0,R)=\{\,y\in T_{P_0}V_\RR\mid \Vert y\Vert_\infty\leq R\,\}.\]
We may then introduce the probability measure
on $\mathcal B(0,R)$ defined by
\[\delta^\alpha_{R,B}=\frac
1{\card(V(\QQ)_{H\leq B}\cap\rho^{-1}(\mathcal B(O,RB^{-\alpha})))}
\sum_{P\in V(\QQ)_{H\leq B}\cap\rho^{-1}(\mathcal B(0,RB^{-\alpha}))}
\delta_{B^\alpha\rho(P)}.
\]
\begin{listrems}
  \remark Let us assume that $P_0$ belongs to a Zariski open subset of~$V$
  on which the rational points of bounded height are equidistributed in the
  sense of~\ref{defi:relativeequidistribution}. For $\alpha=0$,
  we get the measure induced on $B(0,R)$ by $\rho_*(\mu_\infty)$.
  \remark Under the same hypothesis, if $\alpha$ is small, corresponding
  to a small zoom, we may expect that the points are evenly distributed:
  the measure converges to the probability measure induced by the Lebesgue
  measure.
  \remark If $\alpha$ is big enough, diophantine approximation tells
  us that there are no rational point that near to the rational point $P_0$.
  In other words, for $\alpha$ big enough the above measure is the
  Dirac measure at $P_0$. 
\end{listrems}
We are interested in the critical values of $\alpha$, that is those for which
the asymptotic behaviour of the measure $\delta^\alpha_{R,B}$ changes.
In particular, we can consider the smallest value of~$\alpha$ for which
the measure is not the  Dirac measure at~$P_0$, which is the
biggest of the critical values. This is directly
related to the generalisation of the measures of irrationality
introduced by D.~McKinnon and M.~Roth
in~\cite{mckinnonroth:seshadri}. In our context,
with a height defined by an adelic metric on $V$, the archimedean
metric defines a distance $d_\infty$ on $V(\RR)$.
Then if~$W$ is a constructible subset of~$V$ containing~$P_0$,
we define in this text $\alpha_W(P_0)$ as
\[\inf\left\{\,\alpha\in\RRpp\left|
\forall C\in\RR,\left\{Q\in W(\QQ)\left|
d_\infty(Q,P_0)<\frac C{H(Q)^{\alpha}}\right.\right\}
\text{ is finite }\,\right.\right\}.\]
Since $\rho$ is a diffeomorphism,  
$\alpha_V(P_0)$ corresponds to the biggest critical value.
\begin{rema}
  In this text, we take the inverse of the constant defined
  by D.~McKinnon and M.~Roth in their paper (\loccit),
  since it better expresses the power
  appearing in the zoom factor.
\end{rema}
In \cite{mckinnon:conjecture}, D.~McKinnon suggests that there should exist
rational curves~$L$ in~$V$ such that $\alpha_V(P_0)=\alpha_L(P_0)$.
In other words the best approximations should come from
rational curves.
On the other hand D.~McKinnon and
M.~Roth~\cite[theorem 2.16]{mckinnonroth:seshadri}
give the following formula
for $\alpha_L(P_0)$: let $\varphi:\PP^1_\KK\to L$ be a normalisation
of the curve $L$
\[\alpha_L(P_0)=\max_{Q\in\varphi^{-1}(P_0)}\frac{r_Qm_Q}d\]
where $d=\deg(\varphi^*(\omega_V^{-1}))$, $m_Q$ is the multiplicity of
the branch of~$L$ through~$x$ corresponding to~$Q$ and~$r_Q$
corresponds to the approximation of $Q$ by rational points in $\PP^1_\QQ$
and is given by Roth theorem \cite{roth:approximation}:
\[r_Q=
\begin{cases}
  0\text{ if $\kappa(Q)\not\subset\RR$},\\
  1\text{ if $\kappa(Q)=\QQ$},\\
  2\text{ otherwise.}
\end{cases}
\]
On the other hand, if we take a sequence of rational points
$(Q_n)_{n\in\NN}$ on $L(\QQ)$ which converges to $P_0$
then $(H(Q_n))_{n\in\NN}$ goes to $+\infty$ and therefore,
by proposition~\ref{prop:rational},
we have that $(l(Q_n))_{n\in\NN}$ converges to
$l(\varphi)$. In the case where there exists a branch of degree~$1$
through $P_0$, if the deformations of the morphism~$\varphi$ are contained
in a strict subvariety, this means that all the tangent vectors in $T_{P_0}V$
can not be obtained by a deformation of~$\varphi$ and thus $\varphi$ can
not be very free. Under these assumptions, we get that $l(\varphi)\leq 0$
and therefore $(l(Q_n))_{n\in\NN}$ converges to~$0$.
Therefore, if the locally accumulating subvarieties are dominantly covered
by rational curves, we may expect that the freeness of the points on these
locally accumulating subvarieties
tends to~$0$. 

In \cite{huang:six}, \cite{huang:toric}, and \cite{huang:toricII},
\begin{CJK}{UTF8}{gbsn}黄治中\end{CJK}
studies the local distribution of points
on various toric surfaces, exhibiting phenomena
like local accumulating subvarieties, and locally accumulating
thin subsets.
\section{Another description of the slopes}
\begin{cons}
  For any vector bundle~$E$ of rank~$r$ on~$V$, we may define
  the \emph{frame bundle of~$E$}, denoted by
  $F(E)$, as the $GL_r$-torsor of the
  basis in $E$: for any extension $\LL$ of $\KK$ and any
  point $P\in V(\LL)$, the fibre of $F(E)$ at $P$
  is the set of basis of the fibre $E_P$.
  For a line bundle~$L$, the frame bundle $F(L)$ is
  equal to $L^\times$.
  \par
  Let us now assume that~$E$ is equipped with an adelic norm
  $(\Vert\cdot\Vert_w)_{w\in\Val(\KK)}$. Then for any place $w$,
  any point $P\in V(\KK_w)$ and any basis
  $\boldsymbol e=(e_1,\dots,e_r)\in F(E)_P$
  we get an element~$M_w$ in $\GL_r(\KK_w)/K_w$
  where
  \[K_w=
  \begin{cases}
    \GL_r(\mathcal O_w)\text{ if~$w$ is ultrametric,}\\
    O_r(\RR)\text{ if~$w$ is real,}\\
    U_r(\RR)\text{ if~$w$ is complex.}
  \end{cases}
  \]
  which is the class of the matrix of the coordinates of $(e_1,\dots,e_r)$
  in a basis of the $\mathcal O_w$ lattice (resp. orthonornal basis)
  defined by $\Vert\cdot\Vert_w$ if~$w$ is ultrametric
  (resp. non-archimedean).
  We get a map
  \[F(E)(\Adeles_\KK)\longrightarrow \GL_r(\Adeles_\KK)/K,\]
  where~$K$ is the compact subgroup $\prod_{w\in\Val(\KK)}K_w$.
  Taking the quotient by $\GL_r(\KK)$ for the rational points we get
  a map
  \[V(\KK)\longrightarrow \GL_r(\KK)\backslash \GL_r(\Adeles_\KK)/K.\]
  Let us denote by $Q_r$ the biquotient on the right,
  we get a map
  \[\tau_E:V(\KK)\longrightarrow Q_r.\]
  The determinant composed with product of the norms
  gives a morphism of groups from
  the adelic group
  $\GL_n(\Adeles_\KK)$ to $\RRpp$ which is invariant under the action
  of~$K$ on the right and the action of $\GL_n(\KK)$ on the left,
  this gives a map $|\det|:Q_r\to\RRpp$. The composition
  $|\det|\circ\tau_E$ coincides with the exponential
  height $H_E$ defined by~$E$ with
  its adelic norm.
  \par
  Similarly, since the slopes $\mu_i^E$ are defined in terms of the
  $\mathcal O_\KK$-module defined by the norms at the ultrametric places
  equipped with the non-archimedean norms, we may factorise the slopes
  through $Q_r$, and the freeness of a rational point $P$ may also be computed
  in terms of $\tau_{TV}(P)$.
\end{cons}
\begin{listrems}
  \remark
  In $Q_r$, we may consider the subset $Q^1_r$ of points $P$
  such that $|\det|(P)=1$. The d\'eterminant map
  then defines a map $Q^1_r\to\KK^*\backslash\GG_m(\Adeles_\KK)^1/K_{\GG_m}$
  where $K_{\GG_m}$ is the product over the places~$w$
  of the maximal compact subgroup in
  $\GG_m(\KK_w)$. We get a map $c:Q^1_r\to\Pic(\mathcal O_\KK)$; the
  composition map $c\circ\tau_E$ maps a rational point $P$
  onto the class of the projective $\mathcal O_\KK$-module defined by the
  ultrametric norms in $E_P$.
  As an example, for the projective space $\PP^n_\KK$, with $E=TV$,
  this maps a point
  $P=[y_0:\dots:y_n]$ with integral homogeneous coordinates
  to $(n+1)$ times the class of the ideal $(y_0,\dots,y_n)$.
  \remark
  For surfaces, as described in remark~\ref{rema:slopes.surfaces},
  the slopes, and thus the freeness, measures the deformation of the
  lattice or the proximity to the cusp in the modular curve $X(1)$.
  The above construction generalises this description in higher dimension.
  \unskip\remark
  The frame bundle would enable $\GL_n$ descent on varieties for which
  the lifting to versal torsors is not sufficient. In fact we may extend
  this and consider bundles giving geometric elements in
  the Brauer group. This may provide a method to generalise the
  description of Salberger in the case the geometric Brauer
  group is not trivial.
\end{listrems}
\section{Conclusion and perspectives}
In these notes we made a quick survey of the various directions
to upgrade the principle of Batyrev and Manin to include
the cases of Zariski dense accumulating subsets.
Let me summarize these options:
\begin{enumerate}
\item
  Remove accumulating thin subsets. This method has been successful
  in several cases. However, this notion depends on the ground field
  and we could imagine situations in which there are infinitely many thin
  subsets to remove, similar to the situation of K3-surfaces containing
  infinitely many rational lines which are all accumulating.
\item
  Consider all heights. This method may apply to fibrations
  and other cases in which the accumulating subsets come from
  line bundles. However, as shown by examples of Picard rank one,
  this is not enough to detect accumulating subsets of higher codimension.
\item
  As in \cite{peyre:liberte}, we could use a height defined by an
  adelic metric and the freeness. There are no known counter examples,
  but the freeness condition tends to remove too many points as shown
  by the product of projective lines.
\item
  Combine all heights and freeness. This combination is inspired by the
  geometric analogue.
\end{enumerate}
This list is far from exhaustive. In fact, we could consider the slopes
given by norms on any vector bundle on our variety which gives a profusion
of probably redundant invariants. Arakelov geometry is a very natural
tool to attack this question of redundancy and find if there is a minimal
set of slopes controlling the distribution of points.

The freeness, which is in part
suggested by the analogy with the geometry, is very efficient to
detect local adelic deformations which correspond to local or global
accumulation. However this invariant is particularly difficult
to compute efficiently. Indeed its explicit computation
is related to the finding of non-zero vector of minimal length
in a lattice which is known to be computationally difficult.
At the time of writing, the
following question is still open:
\begin{question}
  Let $V$ be a smooth hypersurface of degree~$d$ in $\PP^N_\QQ$,
  with $d\geq 3$ and $N>(d-1)2^{d}$. Is the cardinal of points
  $x\in V(\QQ)$ with $l(x)<\varepsilon$ and $H(x)<B$ negligible
  as~$B$ goes to infinity?
\end{question}
In other words, the author is still lacking methods giving
lower bounds for the smallest slope, but again we may hope that the
techniques of Arakolov geometry may provide the necessary tools.

\providecommand{\noopsort}[1]{}
\ifx\undefined\bysame
\newcommand{\bysame}{\leavevmode\hbox to3em{\hrulefill}\,}
\fi
\ifx\undefined\numero
\newcommand{\numero}{$\hbox{n}^\circ$}
\fi
\ifx\undefined\andname
\newcommand{\andname}{and }
\fi
\ifx\undefined\comma
\newcommand{\comma}{,}
\fi


\end{document}